\pdfoutput=1

\documentclass[12pt]{amsart}
\usepackage{amssymb,latexsym}
\usepackage{enumerate}
\usepackage{amscd}
\usepackage{verbatim}
\usepackage{tabularht}

\theoremstyle{plain}
\newtheorem{theorem}{Theorem}
\newtheorem{corollary}[theorem]{Corollary}
\newtheorem{lemma}[theorem]{Lemma}

\theoremstyle{definition}
\newtheorem{definition}[theorem]{Definition}
\newtheorem{example}[theorem]{Example}

\theoremstyle{remark}
\newtheorem*{remark}{Remark}
\newtheorem*{notation}{Notation}

\numberwithin{theorem}{section}
\def\pair#1#2{\langle #1, #2\rangle}

\def\Mof#1{\Cal M(A)}

\newcommand{\N}{{\mathbb N}}
\newcommand{\Z}{{\mathbb Z}}

\def\Set{{\text{\bf Set}}}
\def\CatFil{{\textbf{\bf Fil}}}

\def\Unif{{\text{\bf Unif}}}
\def\SemiUnif{{\text{\bf SemiUnif}}}

\newcommand{\Ind}{\text{\bf Ind}}

\newcommand{\SuperUnif}{\text{\bf SupUnif}}
\newcommand{\SuperEqv}{\textbf{\bf SupEqv}}
\newcommand{\SemiSuperEqv}{\textbf{\bf SemiSuperEqv}}

\newcommand{\D}{\mathbf d}

\newcommand{\Cg}{\operatorname{Cg}}
\newcommand{\SEg}{\operatorname{SEg}}
\newcommand{\SUg}{\operatorname{SUg}}
\newcommand{\Ug}{\operatorname{Ug}}
\newcommand{\Con}{\operatorname{Con}}

\newcommand{\Eqv}{\operatorname{Eqv}}

\newcommand{\Fg}{\operatorname{Fg}}

\newcommand{\Ig}{\operatorname{Ig}}

\newcommand{\OpSemiUnif}{\operatorname{SemiUnif}}
\newcommand{\OpSemiSuperUnif}{\operatorname{SemiSupUnif}}
\newcommand{\OpSuperUnif}{\operatorname{SupUnif}}

\newcommand{\nat}{\operatorname{nat}}

\newcommand{\Idl}{\operatorname{Idl}}
\newcommand{\Fil}{\operatorname{Fil}}

\newcommand{\nullset}{\{\}}

\newcommand{\I}{\operatorname{I}}
\newcommand{\OpUnif}{\operatorname{Unif}}
\newcommand{\OpSuperEqv}{\operatorname{SupEqv}}
\newcommand{\OpSemiSuperEqv}{\operatorname{SemiSupEqv}}

\newcommand{\Rel}{\operatorname{Rel}}
\newcommand{\tr}{\operatorname{tr}}

\def\congruence{on\-gru\-ence\discretionary{-}{}{-}}
\def\conM/{c\congruence mod\-u\-lar}
\def\ConM/{C\congruence mod\-u\-lar}
\def\conD/{c\congruence dis\-trib\-u\-tive}
\def\ConD/{C\congruence dis\-trib\-u\-tive}
\def\conP/{c\congruence per\-mut\-a\-ble}
\def\ConP/{C\congruence per\-mut\-a\-ble}
\def\conMity/{\conM/\-i\-ty}
\def\ConMity/{\ConM/\-i\-ty}
\def\conDity/{c\congruence dis\-trib\-u\-tiv\-i\-ty}
\def\ConDity/{C\congruence dis\-trib\-u\-tiv\-i\-ty}
\def\conPity/{c\congruence per\-mut\-a\-bil\-i\-ty}
\def\ConPity/{C\congruence per\-mut\-a\-bil\-i\-ty}
\def\usprv/{un\-der\-ly\-ing-set-pre\-ser\-ving}

\def\ie/{{i.e.}}
\def\Ie/{{I.e.}}
\def\eg/{{e.g.}}
\def\Eg/{{E.g.}}
\def\etc/{{etc.}}

\newdimen\mysubdimen
\newbox\mysubbox

\def\subwhat#1#2#3{{
\setbox\mysubbox=\hbox{#3}
\mysubdimen=\wd\mysubbox
\setbox\mysubbox=\hbox{$#1#2$}
\ifnum\mysubdimen>\wd\mysubbox
\vtop{
\hbox to\mysubdimen{\hfil\box\mysubbox\hfil}
\nointerlineskip
\hbox{#3}}
\else
\mysubdimen=\wd\mysubbox
\vtop{
\box\mysubbox
\nointerlineskip
\hbox to\mysubdimen{#3}}
\fi
}}

\begin{document}

% One author
\title[Uniformities, Superequivalences, and  Superuniformities]{Uniformities, Superequivalences, and  Superuniformities\\ of \\Algebras in Congruence-Modular Varieties}
\author{William H. Rowan}
\address{PO Box 20161 \\
         Oakland, California 94620}
\email{william.rowan@ncis.org}
%\thanks{thanks}
% End one author

%\def\cite#1{[#1]}

\date{\today}

\begin{abstract}
We introduce superequivalence and superuniform spaces, and \begin{comment}
investigate commutator operations on compatible uniformities,
superequivalences and superuniformities. We explore the properties of these commutator operations in the congruence-modular case. The commutator operations make sense when applied to equivalence relations, superequivalences, or superuniformities compatible with a subclone of the term operations containing a set of Day operations, yielding binary operations on those lattices which are completely additive but do not satisfy the property that $[x,y]\leq x\wedge y$.
\end{comment}
prove that the lattices of compatible superequivalences and superuniformities on an algebra are modular when the algebra is in a congruence-modular variety.
\end{abstract}

\maketitle

\section*{Introduction}

\begin{comment}
The \emph{commutator} is a binary operation on
congruences in the congruence lattice $\Con A$ of an
algebra $A$ in a congruence-modular variety, and which
is sometimes defined for more general varieties. See \cite{F-McK} for an explanation of the commutator and its applications, as well as some relevant background on Universal Algebra and many references to the literature at the time it was written.  See \cite{KK} for a later perspective on Commutator Theory, including new applications and references to more recent literature.

The commutator is so named because it generalizes the
notion of the commutator of normal subgroups of a group.
(Of course, the variety of groups is
congruence-modular.)  As a further example, the variety
of commutative rings is congruence-modular, and for a
commutative ring
$A$, the commutator is simply the product of ideals. That
is, if
$I_\alpha$ denotes the ideal corresponding to
$\alpha\in\Con A$, then we have
$I_{[\alpha,\beta]}=I_\alpha I_\beta$.
\end{comment}

The purpose of this paper is introduce (in Section~\ref{S:SuperEquivalences}) \emph{superequivalences} and (in Section~\ref{S:SuperUniformities}) \emph{superuniformities}. Like a uniformity on a set, superequivalences and superuniformities are structures which satisfy versions of the reflexive, symmetric, and transitive laws.  For a uniformity, that base structure is a filter of relations on the set; for a superequivalence, it is an ideal in the lattice of relations; for a superuniformity, it is an ideal in the lattice of filters of relations.

\begin{comment}
We show that by relativizing the definitions of these commutators to a subset of the term operations, we obtain a reasonable binary operation on the set of equivalence relations (uniformities, superequivalences, superuniformities) compatible with just those operations, using the example of commutative rings in Section~\ref{S:Field} to make a concrete discussion. In this example, we end up with something resembling fractional ideals and the multiplication of fractional ideals, although as an innovation in Number Theory, we make no claims of usefulness.

[rewrite this paragraph to take such results into account]
Commutator theory (on congruences) works
best for congruences on algebras in congruence-modular
varieties. The same is true of the commutators of
compatible superequivalences and compatible superuniformities described here. In fact, [FIXME is this true?] the
commutator of congruences $\alpha$ and
$\beta$ becomes a special case of that of
compatible superuniformities, when we view
$\alpha$ and
$\beta$ as the compatible superuniformities $\Ig\{\,\Fg\{\,\alpha\,\}\,\}$
and
$\Ig\{\,\Fg\{\,\beta\,\}\,\}$ that they generate, because we have
$\Ig\{\,\Fg\{\,[\alpha,\beta]\,\}\,\}=[\Ig\{\,\Fg\{\,\alpha\,\}\,\},
\Ig\{\,\Fg\{\,\beta\,\}\,\}]$.

\begin{comment}
We follow the development of Commutator Theory in
\cite{F-McK} fairly closely
\end{comment}

 The thesis of
\cite{r02}, where compatible uniformities were first
studied systematically in the context of Universal
Algebra, is that compatible uniformities can be considered
a generalization of congruences. Often, there is
a reasonably direct translation of congruence-theoretic
arguments into uniformity-theoretic ones. This philosophy can be applied as well to super\-equivalences and super\-uniformities.
\begin{comment}
, and using it, we are able to generalize (in
Sections~\ref{S:Term}, \ref{S:Weak}, \ref{S:Commutators}, \ref{S:Clones}, and
\ref{S:Commutator}) the concept of
$C(\alpha,\beta;\delta)$ ($\alpha$ \emph{centralizes}
$\beta$ \emph{modulo} $\delta$) to compatible superequivalences and compatible
superuniformities, and in the congruence-modular case, to
define $[\mathcal U,\mathcal V]$, for compatible uniformities $\mathcal U$ and $\mathcal V$, $[\mathcal I,\mathcal I']$, for compatible superequivalences $\mathcal I$ and $\mathcal I'$, and $[\mathcal E,\mathcal E']$, for compatible superuniformities $\mathcal E$ and $\mathcal E'$, to be the least 
compatible uniformity $\mathcal W$ (respectively, compatible superequivalence $\mathcal I''$, compatible superuniformity $\mathcal E''$) such that $C(\mathcal U,\mathcal V;\mathcal W)$ (respectively, $C(\mathcal I,\mathcal I';\mathcal I'')$, $C(\mathcal E,\mathcal E';\mathcal E'')$).
\end{comment}
We find that the philosophy works even better for superequivalences and  superuniformities than it does for uniformities, because of the fact that the lattices involved are algebraic.

The cases of compatible  equivalence relations or congruences (see \cite{F-McK} regarding modularity) superequivalences (Section~\ref{S:SuperEquivalences}), uniformities (\cite{r02}, Section~\ref{S:Unif}), and superuniformities (Section~\ref{S:SuperEquivalences}) are shown in the following table having four rows, which could be extended to an infinite number of rows summarizing the principal answer of this paper with respect to arbitrary iterates of the functors $\Fil$ (filters) and $\Idl$ (ideals), applied to the lattices of relations on a set:
\begin{center}
  \begin{tabular}{ | l | c | l | c | }
    \hline
    Lattice & Algebraic & Compatible Equivalence & Modular \\ \hline
    $\Rel S$ & yes & congruence & yes \\ \hline
    $\Idl\Rel S$ & yes & compatible  superequivalence & yes \\ \hline
    $\Fil\Rel S$ & no & compatible uniformity & no \\ \hline
    $\Idl\Fil\Rel S$ & yes &  compatible superuniformity & yes \\
  	\hline
  \end{tabular}
\end{center}
where a \emph{compatible equivalence} (third column) is an element of the lattice (first column) satisfying generalized versions of the reflexive, symmetric, and transitive laws, and also, compatible with the operations of the algebra.

Thus, this paper is essentially a study of the two new rows in the table. For a brief discussion of the general theorem hinted at by the table, see Section~\ref{S:Pattern}. As to the usefulness of this information, we discuss that briefly in Section~\ref{S:Conclu}, Conclusions.

\begin{comment}
In the congruence-modular case, the properties of
the commutator on compatible uniformities described in
Section~\ref{S:Properties} duplicate those of the
commutator of congruences, with some regrettable gaps such as the lack of a proof of complete additivity.  However, using the algebraicity of the lattices involved, we do  prove complete additivity for commutators of compatible superequivalences and compatible superuniformities.

[rewrite this later after content is settled, and perhaps move it to a summary of the paper]
Section~\ref{S:Examp} is devoted to miscellaneous matters,
including commutators of congruential uniformities on
commutative rings. We prove that in that case, the two
definitions of the commutator on congruential
uniformities, the one given by the general definition of
Section~\ref{S:Commutator} and the other given by the
formula of Section~\ref{S:Congruential}, coincide. This
appears to be a special property of commutative rings; in
general, we do not know even whether the commutator
operation of Section~\ref{S:Commutator}, applied to
congruential uniformities, always gives a congruential
uniformity. We also discuss in this section
the case of varieties which are
congruence-distributive, where we show that as in the
case of the commutator of congruences, the commutator of
two compatible uniformities is simply their meet.

[rewrite]
In the last section, we
discuss the current state of some questions about
compatible uniformities, uniformity lattices, and
commutators of uniformities.
\end{comment}

\section*{Preliminaries}

In this section are some preliminaries providing some explanations of concepts which may be unfamiliar. The reader may also want to consult references about Universal Algebra such as \cite{B-S} and \cite{McK-McN-T}.

\subsection{Notation}
The composition of relations $R$ and $R'$ will be denoted by $R\circ R'$. For associative operations such as $\circ$, we will often denote $R\circ R$ by $R^{\circ 2}$, $R\circ R\circ R$ by $R^{\circ 3}$, etc. The opposite of a relation $R$ will be denoted by $R^{-1}$.
The identity relation on a set $S$ will be denoted by $\Delta_S$, or simply $\Delta$.

\subsection{Category theory} We follow \cite{macl} in
terminology and notation.  In particular, $1_A$ will stand
for the identity arrow on an object $A$ in a category
$\mathbf C$.  For example, $A$ could be an algebra, and $\mathbf C$ could be the category of algebras in a variety $\mathbf V$ containing $A$. (On the other hand, if $1$ is a constant symbol for an algebra $A$ of some type of algebras, we will denote the $1$ element of $A$ by $1^A$.)

\subsection{Lattice theory} The reader should be
familiar with lattices. We use $\top^L$ and
$\bot^L$, or just $\top$ and $\bot$, to denote the greatest and least elements of a
lattice $L$, assuming they exist, and
$\wedge^L$ and
$\vee^L$, or just $\wedge$ and $\vee$, for the meet and join operations.

\subsection{Compatible relations}

Suppose $A$ is an algebra, and $R$ is a relation on $|A|$, the underlying set of $A$. We say that $R$ is compatible if for every operation $w$ of the algebra $A$, which is, say, $n$-ary, and every $n$-tuple of pairs $\pair{x_i}{y_i}$ such that $x_i\mathrel Ry_i$, we have $w(x_1,\ldots,x_n)\mathrel Rw(y_1,\ldots,y_n)$.

\subsection{Monotone Galois connections}

A (monotone) Galois connection between lattices $L_1$ and $L_2$ is a pair of monotone functions $f:L_1\to L_2$, $g:L_2\to L_1$ such that $f(x)\leq y$ iff $x\leq g(y)$. We say that $f$ is the \emph{left adjoint} of $g$, and $g$ is the \emph{right adjoint} of $f$. We say
\[\pair fg:L_1\rightharpoonup L_2,\]
borrowing the notation used for pairs of adjoint functors.  Note that unlike with adjoint functors, when we talk about monotone Galois connections, left adjoints are unique, as are right adjoints.

The left adjoint preserves all joins which exist in $L_1$ and the right adjoint preserves all meets which exist in $L_2$. Conversely, if $f:L_1\to L_2$ is a monotone function which preserves all joins, and $L_1$ is complete, then we can define a right adjoint $g:L_2\to L_1$ by $g:y\mapsto\bigvee_{f(x)\leq y}x$, and dually, given a monotone function $g:L_2\to L_1$ that preserves all joins, and if $L_2$ is complete, we can define a left adjoint $f:L_1\to L_2$ by $f:x\mapsto\bigwedge_{x\leq g(y)}y$. This may be compared with the Special Adjoint Functor Theorem: \cite[Theorem~V.8.2]{macl}.

A first example of a monotone Galois connection: given sets $S$ and $T$ and a function $f:S\to T$, we have
\[\pair{[R\mapsto f(R)]}{[R\mapsto f^{-1}(R)]}:\Rel S\rightharpoonup \Rel T;\]
and we will see many more examples in Sections~\ref{S:Unif}, \ref{S:SuperEquivalences}, and~\ref{S:SuperUniformities}, where we have used the theory of monotone Galois connections as an organizing principle. See \cite{DP} for more about monotone Galois connections.

\subsection{Filters}

If $L$ is a lattice, then a nonempty subset $F\subseteq
L$ is called a \emph{filter} if $y\geq x\in F$ implies
$y\in F$ and $x$, $y\in F$ imply $x\wedge y\in F$.

If $S\subseteq L$ is a nonempty set, then the
\emph{filter generated by $S$}, denoted by $\Fg^L S$
or simply $\Fg S$, is the subset of elements of $L$ that
are greater than or equal to a finite meet of elements of
$S$. An important special case, given $x\in L$, is
$\Fg\{\,x\,\}$, the
\emph{principal filter generated by $x$}, which is just
the set of elements of $L$ greater than or equal to $x$.

We order the set of filters by reverse inclusion (for formal consistency between the theories of congruences and compatible uniformities) and the filters
in a nonempty lattice form a complete lattice. The meet
of a tuple of filters $F_i$ is given
by $\bigwedge_iF_i=\Fg(\bigcup_iF_i)$. The join of the
tuple is the intersection: $\bigvee_iF_i=\bigcap_iF_i$.

If $F$ is a filter, a \emph{base} for $F$ is a subset
$B\subseteq F$ such that $x\in F$ implies $b\leq x$ for
some $b\in B$. If $L$ is a lattice, then a subset
$B\subseteq L$ is a base for a filter of $L$, or
\emph{filter base}, iff given any $x$, $y\in B$, there is
a
$z\in B$ such that
$z\leq x\wedge y$.

\subsection{Ideals}

The concept of an \emph{ideal} in a  lattice is dual to that of a filter. Thus,
if $L$ is a lattice, a nonempty subset $J\subseteq
L$ is called an \emph{ideal} if $x\geq y\in J$ implies
$y\in J$ and $x$, $y\in J$ imply $x\vee y\in F$.

If $S\subseteq L$ is a nonempty set, then the
\emph{ideal generated by $S$}, denoted by $\Ig^L S$
or simply $\Ig S$, is the subset of elements of $L$ that
are less than or equal to a finite join of elements of
$S$. An important special case, given $x\in L$, is
$\Ig\{\,x\,\}$, the
\emph{principal ideal generated by $x$}, which is just
the set of elements of $L$ less than or equal to $x$.

The set of ideals is ordered by inclusion and the ideals
in a nonempty lattice form a complete lattice. The meet
of a tuple of ideals $F_i$ is given
by $\bigwedge_iJ_i=\bigcap_iF_i$. The join of the
tuple is $\bigvee_iJ_i=\Ig\{\,\bigcup_iJ_i\,\}$.

We call the notion, dual to a base for a filter, a \emph{ceiling of an ideal}. Every \emph{directed set} (set $S\subseteq L$ such that if $x$, $y\in S$, there is a $z\in S$ such that $x\leq z$ and $y\leq z$) is a ceiling for an ideal.

\subsection{Universal algebra}

We assume familiarity with universal algebra, as
explained for example in \cite{B-S}.  We prefer to allow
an algebra to have an empty underlying set, however. We
denote the underlying set of an algebra $A$ by $|A|$.

If $R$ is a binary relation on (the underlying
set of) an algebra $A$, then we denote by
$\Cg R$ the smallest congruence $\alpha\in\Con A$ such
that
$R\subseteq\alpha$.

We will call a compatible, reflexive, symmetric relation on an algebra a \emph{tolerance}.

\subsection{Congruence-permutable varieties}

A variety of algebras $\mathbf V$ is \emph{congruence-permutable} (or, a \emph{Mal'tsev variety}) if for every $A\in\mathbf V$, for every $\alpha$, $\beta\in\Con A$, $\alpha\circ\beta=\beta\circ\alpha$. A variety is congruence-permutable iff \cite{Mal} there is a ternary term $p$, called a \emph{Mal'tsev term}, satisfying the identities $p(x,x,y)=y$ and $p(x,y,y)=x$. For example, the variety of groups is congruence-permutable because it has the Mal'tsev term $p(x,y,z)=xy^{-1}z$.

\subsection{Congruence-modular varieties}

A
variety of algebras
$\mathbf V$ is
\emph{congruence-modular} if for every algebra $A\in\mathbf V$,
$\Con A$ is a modular lattice.
A theorem
\cite{Day} states that a variety
$\mathbf V$ is congruence-modular iff there is a natural number number $\D$ and a sequence $m_0$, $\ldots$, $m_\D$ of quaternary
terms, called \emph{Day terms}, satisfying the following
identities:
\begin{enumerate}
\item[(D1)] For all $i$, $m_i(x,y,y,x)=x$;
\item[(D2)] $m_0(x,y,z,w)=x$;
\item[(D3)] $m_\D(x,y,z,w)=w$;
\item[(D4)] for even $i<\D$,
$m_i(x,x,y,y)=m_{i+1}(x,x,y,y)$; and
\item[(D5)] for odd $i<\D$,
$m_i(x,y,y,z)=m_{i+1}(x,y,y,z)$.
\end{enumerate}

\begin{example} \label{E:Maltsev}
A congruence-permutable variety, with Mal'tsev term $p$,
is necessarily congruence-modular, with one possible sequence of Day terms given by $\D=2$ and
\begin{align*}
m_0(x,y,z,w)&=x;\\
m_1(x,y,z,w)&=p(x,p(z,y,x),w);\\
m_2(x,y,z,w)&=w.\end{align*}
Note that for the variety of groups, when we use the
Mal'tsev term
$p(x,y,z)=xy^{-1}z$, this gives $m_1(x,y,z,w)=yz^{-1}w$.
\end{example}

A key theorem utilizing Day terms is the following:

\begin{theorem}[{\cite[Lemma~2.3]{F-McK}}]\label{T:KeyLemma}
Let $\mathbf V$ be a variety having a sequence $\mathbf m$ of Day terms, and let $A\in\mathbf V$, $\gamma\in\Con A$, and $a$, $b$, $c$, $d\in A$ with $b\mathrel\gamma d$. Then $a\mathrel\gamma c$ iff for all $i$, $m_i(a,a,c,c)\mathrel\gamma m_i(a,b,d,c)$.
\end{theorem}

\begin{comment}
Later on we will sometimes assume that a congruence-modular variety $\mathbf V$ has a distinguished sequence $\mathbf m$ of Day terms. \end{comment}
It is not true that there is always just one set of Mal'tsev terms or one possible set of Day terms:

\begin{example}
Consider the variety $\mathbf V$ having a type with two sets of group operations, and group identities within each set, but no identities connecting the basic operations of the two sets. There are many models, e.g. the set $A$ of integers from $0$ to $999$ with two group structures: addition modulo $1000$, and a randomly-chosen abelian group structure for a set of $1000$ elements. It seems the usual Mal'tsev terms have little relation, at least, as operations on $A$, and the same is true of the Day terms based on them as given in Example~\ref{E:Maltsev}.
\end{example}

\begin{comment}
When we have an algebra $A$ in a congruence-modular variety $\mathbf V$ with distinguished sequence $\mathbf m$ of Day terms, we will denote by $A_{\mathbf m}$ the underlying set of $A$, provided with the term operations $m_i$ as basic operations.  For example, if $R$ is a commutative ring, and we choose the sequence of Day terms $\mathbf m=\langle x,y-z+w,w\rangle$, then $A_{\mathbf m}$ is the underlying abelian heap (i.e., the underlying abelian group, with the zero constant operation forgotten).
\end{comment}

Our treatments of uniformities, superequivalences, and superuniformities will not cover completion, which is a standard part of the theory of uniform spaces (and covered in \cite{r02}), because completion is not yet worked out completely for superequivalences and superuniformities. We hope to address this lack in the not-too-distant future, because completion of uniform spaces turns out to be the analog of formation of quotient algebras in ordinary Universal Algebra.

\begin{comment}
\subsection{Commutator theory}

If
$A$ is an algebra in a congruence-modular variety, and
$\alpha$,
$\beta\in\Con A$, we can define the
commutator $[\alpha,\beta]$ as the least
$\delta\in\Con A$ such that $\alpha$ \emph{centralizes}
$\beta$ \emph{modulo}
$\delta$, or in other words, such that $\delta$ satisfies
the \emph{$\alpha,\beta$-term condition}. See the first
part of Section~\ref{S:Term} for detailed definitions. 

While Sections~\ref{S:Term}, \ref{S:Weak}, \ref{S:Commutators}, and \ref{S:Clones} give these definitions in
the general case, and stop short of defining
$[\alpha,\beta]$ or its generalization to uniformities, superequivalences, or superuniformities,
Section~\ref{S:Commutator} takes into account the assumption of
congruence-modularity and proves the simplifications
that make the definition of
$[\alpha,\beta]$ so reasonable, and which apply as well to the other cases.
\end{comment}

\subsection{Uniform, superequivalence, and superuniform universal algebra}
Universal algebra over the base category $\Unif$ of
uniform spaces, as opposed to the category of sets, was
first studied systematically in
\cite{r02}. We introduce here the category $\SuperEqv$ of superequivalence spaces and the category $\SuperUnif$ of superuniform spaces, and do a bit of Universal Algebra in those categories.   It will be best if the reader has
access to \cite{r02} while reading this paper, but we also devote the next section to some of the basic definitions and results about $\Unif$.

\section{Uniformities}
\label{S:Unif}

We denote the set of binary relations on a set $S$ by
$\Rel S$. $\Rel S$, ordered by inclusion, is a complete
lattice.

Consider the conditions on a filter
$\mathcal U$ of relations on $S$:
\begin{enumerate}
\item[($\mathrm{U_r}$)] if $U\in\mathcal U$, then
$\Delta_S\overset{\text{def}}=\{\,\pair ss\mid s\in
S\,\}\subseteq U$;
\item[($\mathrm{U_s}$)] if $U\in\mathcal U$, then $U^{-1}\in\mathcal
U$; and
\item[($\mathrm{U_t}$)] if $U\in\mathcal U$, then $V^{\circ 2}\subseteq
U$ for some $V\in\mathcal U$.
\end{enumerate}
We say that $\mathcal U$ is \emph{reflexive} if it satisfies condition ($\mathrm{U_r}$), \emph{symmetric} if it satisfies condition ($\mathrm{U_s}$), \emph{transitive} if it satisfies condition ($\mathrm{U_t}$), a \emph{semiuniformity}
if $\mathcal U$ satisfies conditions ($\mathrm{U_r}$) and ($\mathrm{U_s}$),
and a \emph{uniformity} if it satisfies ($\mathrm{U_r}$), $(\mathrm{U_s}$), and ($\mathrm{U_t}$).

\begin{theorem}\label{T:UB}
A filter base $\mathcal B$ of relations on a set $S$ is a base for a uniformity iff
\begin{enumerate}
\item[$\mathrm{(BU_r)}$] for all $U\in\mathcal B$, $\Delta_S\subseteq U$;
\item[$\mathrm{(BU_s)}$] for all $U\in\mathcal B$, there is a $V\in\mathcal B$ such that $V^{-1}\subseteq U$; and
\item[$\mathrm{(BU_t)}$] for all $U\in\mathcal B$, there is a $V\in\mathcal B$ such that $V^{\circ 2}\subseteq U$.
\end{enumerate}
\end{theorem}

\begin{theorem} A filter $\mathcal U$ satisfies $\mathrm{(U_r)}$ iff every relation in $\mathcal U$ is reflexive.
\end{theorem}

\begin{theorem} A filter $\mathcal U$ satisfies $\mathrm{(U_s)}$
iff
$\mathcal U=\mathcal U^{-1}$, where $\mathcal
U^{-1}=\{\,U^{-1}\mid U\in\mathcal U\,\}$.
\end{theorem}

\begin{theorem} If $\mathcal U$, $\mathcal V$ are
filters in $\Rel S$, then $\{\,U\circ V\mid U\in\mathcal
U,\,V\in\mathcal V\,\}$ is a base for a filter $\mathcal
U\circ\mathcal V$ in $\Rel S$.
\end{theorem}

Note that for filters $\mathcal U$ and $\mathcal V$ of
reflexive relations, $\mathcal U\cap\mathcal
V\leq\mathcal U\circ\mathcal V$. Thus, $\mathcal
U\leq\mathcal U\circ\mathcal U$ if $\mathcal U$ satisfies
$\mathrm{(U_r)}$.

\begin{theorem} A filter $\mathcal U$ satisfies $\mathrm{(U_t)}$
iff
$\mathcal U\circ\mathcal U\leq\mathcal U$.
\end{theorem}

\begin{notation} If $U\in\mathcal U$, where $\mathcal U$
satisfies $\mathrm{(U_t)}$, and $n>0$, then by induction we can show that there
is a $V\in\mathcal U$ such that $V^{\circ n}\subseteq U$. We denote such a
$V$ by
$^nU$. This notation must be used with care, particularly
in relation to quantifiers; it is simply a shorthand for the statement that
there exists such a $V$ and that we will denote one such $V$ by
$^nU$.
\end{notation}

\subsection{The categories $\Unif$ and $\SemiUnif$}

Suppose we consider pairs $\pair S{\mathcal F}$ consisting of a set $S$ and a filter of relations on that set. If $\pair S{\mathcal F}$, $\pair T{\mathcal F'}$ are two such pairs, and $f:S\to T$ is a function, then we denote by $f^{-1}(\mathcal F')$ the filter of relations $\Fg\{\,f^{-1}(R)\mid R\in\mathcal F'\,\}$. We denote by $f(\mathcal F)$ the filter of relations on $T$ generated by the sets $f(R)$, $R\in\mathcal F$, and we have $\pair{[\mathcal F\mapsto f^{-1}(\mathcal F)]}{\mathcal F\mapsto f(\mathcal F)]}:\Fil\Rel S\rightharpoonup\Fil\Rel T$.

\begin{definition} A \emph{uniform space} (\emph{semiuniform space}) is a pair $\pair S{\mathcal U}$ such that $\mathcal U$ is a uniformity (respectively, semiuniformity) on the set $S$.
\end{definition}

\begin{definition}If $\pair S{\mathcal U}$ and $\pair T{\mathcal V}$ are uniform (semiuniform) spaces, and $f:S\to T$, then we say that $f$ is \emph{uniform} if the two equivalent conditions $\mathcal U\leq f^{-1}(\mathcal V)$, $f(\mathcal U)\leq\mathcal V$ are satisfied.
\end{definition}

The uniform (semiuniform) spaces, and the uniform functions between them, form a category, $\Unif$ (respectively, $\SemiUnif$).

\begin{example}
\label{E:CHSpace} Suppose that $S$ is a set, provided with a compact, hausdorff topology $\mathcal T$. Then there is a unique uniformity $\mathcal U(\mathcal T)$ on $S$ having $\mathcal T$ as its topology.

\begin{proof}
Suppose $\pair S{\mathcal U_1}$ and $\pair S{\mathcal U_2}$ are two uniform spaces with underlying topology $\mathcal T$, and consider the identity function $1_S$. It is continuous from $\pair S{\mathcal T}$ to $\pair S{\mathcal T}$,
 hence uniformly continuous from $\pair S{\mathcal U_1}$ to $\pair S{\mathcal U_2}$ and vice versa, because continuous functions are uniformly continuous on
 compact sets. Thus, $\mathcal U_1\leq 1_S^{-1}(\mathcal U_2)=\mathcal U_2$, and similarly, $\mathcal U_2\leq\mathcal U_1$.
\end{proof}
\end{example}

\subsection{The lattices $\OpUnif S$ and $\OpSemiUnif S$}
We denote the set of uniformities on a set $S$ by
$\OpUnif S$, and the set of semiuniformities by
$\OpSemiUnif S$. We order
these sets by reverse inclusion, i.e., the ordering
inherited from
$\Fil\Rel S$.

The join of a tuple of
semiuniformities is simply the intersection, and
$\OpSemiUnif S$ is a distributive lattice. The theory of
joins of uniformities is more difficult, as we shall see.

The meet of an arbitrary tuple of uniformities on $S$, in
the lattice
$\Fil\Rel S$, is a uniformity. Thus, the subset $\OpUnif S\subseteq\Fil\Rel S$ is closed under
arbitrary meets, and is a complete lattice with the inherited ordering.
The same is
true for $\OpSemiUnif S$.

\subsection{$\Ug\mathcal F$}
If $\mathcal F$ is a filter of relations on a set $S$, then we define the \emph{uniformity generated by $\mathcal F$}, as $\Ug\mathcal F=\bigwedge_{\mathcal F\leq\mathcal U\in\OpUnif S}\mathcal U$, and we have $\pair{[\mathcal F\mapsto\Ug\mathcal F]}{[\mathcal U\mapsto\mathcal U]}:\Fil\Rel S\rightharpoonup\OpUnif S$.  (We may use the notation $\Ug^S\mathcal F$ to emphasize that we want to consider the least \emph{not-necessarily-compatible} uniformity on $S$ larger than $\mathcal F$. We discuss \emph{compatible} uniformities on an algebra a bit later.)

\subsection{Permutability} Permutability of congruences
is an important condition in Universal Algebra, and the theory of this
condition generalizes easily to uniformities.

\begin{lemma}\label{T:Permutability}
Let $\mathcal U$, $\mathcal V$ be uniformities on a set $S$.
Then:
 $\mathcal U\circ\mathcal V\leq\mathcal V\circ\mathcal U$ iff
 $\mathcal V\circ\mathcal U\leq\mathcal U\circ\mathcal V$.
\end{lemma}
\begin{proof} We have

\begin{align*}
\mathcal U\circ\mathcal V\leq\mathcal V\circ\mathcal U&\implies
(\mathcal U\circ\mathcal V)^{-1}\leq(\mathcal V\circ\mathcal U)^{-1}\\
&\implies 
{\mathcal V}^{-1}\circ{\mathcal U}^{-1}\leq
{\mathcal U}^{-1}\circ{\mathcal V}^{-1}\\
&\implies\mathcal V\circ\mathcal U\leq\mathcal U\circ\mathcal V;
\end{align*}
the converse holds by symmetry.

\end{proof}

\begin{definition} We say that filters of relations $\mathcal F$, $\mathcal G$ \emph{permute} if $\mathcal F\circ\mathcal G=\mathcal G\circ\mathcal F$.
\end{definition}

The join operation in the lattice of
uniformities may be difficult to deal with in the general
case, but the case where $\mathcal U$ and $\mathcal V$
permute is an easy and important one:

\begin{theorem} Let $\mathcal U$, $\mathcal V\in\OpUnif
S$. Then $\mathcal U\vee\mathcal V=\mathcal
U\circ\mathcal V$ iff $\mathcal U$ and $\mathcal V$ permute.
\end{theorem}

\begin{proof} ($\impliedby$): It is trivial that $\mathcal
U\circ\mathcal V$ is a filter and satisfies $\mathrm{(U_r)}$. If
$U\in\mathcal U$ and $V\in\mathcal V$, then since
$\mathcal V\circ\mathcal U\leq\mathcal U\circ\mathcal
V$, there are $\bar U\in\mathcal U$ and $\bar
V\in\mathcal V$ such that $\bar V\circ\bar U\subseteq
U\circ V$. But, $\bar V\circ\bar U=((\bar
U^{-1})\circ(\bar V^{-1}))^{-1}$. Thus, $\mathcal
U\circ\mathcal V$ also satisfies $\mathrm{(U_s)}$.
Finally, $(\mathcal U\circ\mathcal V)\circ(\mathcal
U\circ\mathcal V)=\mathcal U\circ(\mathcal V\circ\mathcal U)\circ\mathcal V\leq(\mathcal U\circ\mathcal
U)\circ(\mathcal V\circ\mathcal V)\leq\mathcal
U\circ\mathcal V$, verifying $\mathrm{(U_t)}$. Thus, $\mathcal
U\circ\mathcal V$ is a uniformity. Since $\mathcal
U\leq\mathcal U\circ\mathcal V$ and $\mathcal
V\leq\mathcal U\circ\mathcal V$, we have $\mathcal
U\vee\mathcal V\leq\mathcal U\circ\mathcal V$.
However, $\mathcal U\circ\mathcal V\leq\mathcal U\vee
\mathcal V$ by $\mathrm{(U_t)}$.

 ($\implies$): If $\mathcal U\vee\mathcal V=\mathcal
U\circ\mathcal V$, then
$\mathcal V\circ\mathcal
U\leq\mathcal U\vee\mathcal V=\mathcal U\circ\mathcal V$, implying by Lemma~\ref{T:Permutability} that $\mathcal U$ and $\mathcal V$ permute.
\end{proof}

\subsection{$\Ug\mathcal S$ when $\mathcal S$ is a semiuniformity}
Based on \cite[Theorem~2.2]{weber}, we want to present a formula for $\Ug\mathcal S$ when $\mathcal S$ is a semiuniformity. First, a notation;

\begin{notation}[See {\cite[Notation~2.1]{weber}}]
If $\{\,U_n\,\}_{n\in\mathbb N_+}$ is a sequence of relations on a set $S$, then let
\[[\,U_n\,:\,n\in\mathbb N_+]=\bigcup_{n=1}^\infty\bigcup_{\gamma\in\Gamma_n}U_{\gamma(1)}\circ\ldots\circ U_{\gamma(n)},\]
where $\Gamma_n$ denotes the set of permutations of $\{\,1\,\ldots\,n\,\}$.
\end{notation}

\begin{theorem}\label{T:WTheorem}
Let $\mathcal S$ be a semiuniformity on a set $S$. Then the sets
\[[\,U_n\,:\,n\in\mathbb N_+\,],\]
where $U_n\in\mathcal S$ for $n\in \mathbb N_+$, form a base for the uniformity $\Ug\mathcal S$.
\end{theorem}

\begin{proof}
It is clear that the given set of relations is a filter base.

We prove the conditions (Theorem~\ref{T:UB}) for this set to be a base for a uniformity:

$\mathrm{(BU_r)}$: Since every $U\in\mathcal S$ is reflexive by $(U_r)$, so is every relation of the form $[\,U_n\,:\,n\in\mathbb N_+\,]$ for a sequence $\{\,U_n\,\}_{n\in\mathbb N_+}$ of relations $U_n\in\mathcal S$.

$\mathrm{(BU_s)}$: Given $\{\,U_n\,\}_{n\in\mathbb N_+}$, we have the sequence $\{\,U_n^{\mathrm{op}}\,\}_{n\in\mathbb N_+}$, with $U_n^{\mathrm{op}}\in\mathcal S$ by $\mathrm{(U_s)}$, and we have
\[[\,U_n^{\mathrm{op}}\,:\,n\in\mathbb N_+\,]^{\mathrm{op}}\subseteq[\,U_n\,:\,n\in\mathbb N_+\,].\]

$\mathrm{(BU_t)}$: Given $\{\,U_n\,\}_{n\in\mathbb N_+}$, with $U_n\in\mathcal S$, we let $V=[\,U_{2n}\cap U_{2n-1}\,:\,n\in\mathbb N_+\,]$; we have $V^{\circ2} \subseteq [\,U_n\,:\,n\in\mathbb N_+\,]$.

Thus, the formula gives a uniformity $\mathcal U$. We have $\mathcal S\leq\mathcal U$, because given a sequence $\{\,U_n\,\}_{n\in\mathbb N_+}$, we have $U_1\subseteq [\,U_n\,:\,n\in\mathbb N_+\,]$. Suppose then that $\mathcal V$ is a uniformity with $\mathcal S\leq\mathcal V$, and let $V\in\mathcal V$. Then $V\in\mathcal S$; let $V_0=V$, $\ldots$, $V_n={^3V_{n-1}}$. We have
$[\,V_n\,:\,n\in\mathbb N_+\,]\subseteq V$, showing that $\mathcal U\leq\mathcal V$.
\end{proof}

\subsection{Joins} The formula for $\Ug\mathcal S$ can be applied to give a formula for the join of a tuple of uniformities, since if $\mathcal U_i$ are uniformities, $\bigcap_i\mathcal U_i$ will be a semiuniformity:

\begin{theorem} If $\mathcal U_i$, $i\in I$ are
elements of $\OpUnif S$, then $\bigvee_i\mathcal
U_i=\Ug(\bigcap_i\mathcal U_i)$.
\end{theorem}

\subsection{Compatible uniformities} If $R$ is a relation on an
algebra $A$, we say that $R$ is \emph{compatible} (with
the operations of $A$) if $\mathbf a\mathrel R\mathbf
a'$ implies $\omega^A(\mathbf a)\mathrel R
\omega^A(\mathbf a')$ for each operation symbol $\omega$,
where
$\mathbf a\mathrel R\mathbf{a'}$ means that $a_i\mathrel R
a'_i$ for all $i$.

We say
that a filter
$\mathcal U$ of reflexive relations on an algebra $A$ is
\emph{compatible} if for each $U\in\mathcal U$, and
each basic operation symbol $\omega$,
there is a $\bar U\in\mathcal U$ such
that $\omega(\bar U)=\{\,\pair{\omega(\mathrm
x)}{\omega(\mathbf y)}\mid x_i\mathrel{\bar
U}y_i\text{ for all }i\,\}\subseteq U$. In this case, for
any term
$t$, given $U\in\mathcal U$, there is a $\bar
U\in\mathcal U$ such that $t(\bar U)\subseteq U$.

If $A$ is an algebra, then by $\OpUnif A$ ($\OpSemiUnif A$) we will mean the lattice of compatible uniformities (respectively, semiuniformities). If instead we want to talk about the lattice of not-necessarily compatible uniformities or semiuniformities, we will say $\OpUnif|A|$ or $\OpSemiUnif|A|$.
An arbitrary meet of compatible uniformities (semiuniformities) is compatible.
Thus, the subset of $\OpUnif|A|$ of compatible uniformities is a complete lattice,

\subsection{$\Ug^A\mathcal F$} If $A$ is an algebra, and $\mathcal F$ is a filter of relations on $A$, we define $\Ug^A\mathcal F=\bigwedge_{\mathcal F\leq\mathcal U\in\OpUnif A}\mathcal U$, and we then have $\pair{[\mathcal F\mapsto\Ug^A\mathcal F]}{[\mathcal U\mapsto\mathcal U]}:\Fil\Rel S\rightharpoonup\OpUnif A.$

\begin{theorem} Let $A$ be an algebra.
The meet and join (in the lattice of uniformities on the set $|A|$) of a tuple of
compatible uniformities are compatible.
\end{theorem}

\begin{proof} See \cite[Theorem~5.3]{r02} (for the join; we discussed the meet above) or \cite[Corollary~2.4]{weber}.
\end{proof}

Thus, $\OpUnif A$ is a complete sublattice of $\OpUnif|A|$.

As we mentioned previously, $\OpSemiUnif A$ is also a complete lattice.

\subsection{Examples of compatible uniformities}

\begin{example}
Suppose $\pair A{\mathcal T}$ is a topological algebra, which as a topological space is compact and hausdorff. We assume that $\mathcal T$ is compatible with the operations of $A$. The unique uniformity $\mathcal U(\mathcal T)$ giving rise to $\mathcal T$ (Example~\ref{E:CHSpace}) is compatible.

\begin{proof} Continuous operations are uniformly continuous on compact sets.
\end{proof}
\end{example}

\begin{example} Suppose
$\mathcal R=\Fg\{\,\alpha\,\}$ where $\alpha\in\Con A$. Then
$\Ug^A\mathcal R=\Fg^{\Rel A}\{\,\alpha\,\}$.
\end{example}

\begin{comment}
More generally, we can consider $\Ug\{\,\rho\,\}$ where
$\rho\in\Rel A$. However, we have

[here: note $\Ug$ is smallest \emph{compatible}
uniformity. If we are talking about $\Ug^{|A|}\rho$, then
yes, indeed $\Ug\{\,\rho\,\}$ is the same as
$\Ug$ of the equivalence relation generated by $\rho$]

Recall $\Cg\rho$ is the congruence \emph{generated by} $\rho$.
\begin{theorem} $\Ug\{\,\rho\,\}=\Ug\{\,\Cg\rho\,\}$, and
is compatible if $\rho$ is.
\end{theorem}

\begin{proof}
It suffices to show that
$\Ug\{\,\Cg\rho\,\}\subseteq\Ug\rho$, or in other words
that if $U\in\Ug\{\,\rho\,\}$, then
$\Cg\{\,\rho\,\}\subseteq U$.

Let $U\in\Ug\{\,\rho\,\}$. We have $\rho\subseteq U$, so
$\rho\cup\Delta\subseteq U$ by (U3). Thus, we can reduce
to the case where $\rho$ is reflexive by replacing $\rho$
with $\rho\cup\Delta$.

If $U\in\Ug\{\,\rho\,\}$, then $\rho\subseteq U^{-1}$, so
$\rho^{-1}\subseteq U$. Thus, we can further reduce to
the case where $\rho$ is symmetric, by replacing $\rho$
by $\rho\cup\rho^{-1}$.

Finally, if $U\in\Ug\{\,\rho\,\}$, then $\rho\subseteq
{^nU}$ for all $n\in\mathbb N$, which implies
$\rho\subseteq U$, showing that
$\Cg\rho=\bigcup_n\rho^n\subseteq U$.

As regards compatibility, it is easy to prove that if
$\rho$ is compatible, then so is $\Cg\rho$. It is obvious
that $\Ug\{\,\alpha\,\}$ is compatible if $\alpha$ is a
congruence.
\end{proof}
\end{comment}

\begin{example} More generally, suppose $\mathcal R=\mathcal F$, where
$\mathcal F$ is a filter in $\Con A$. In this case, $\Ug\mathcal
R=\mathcal F$. We call uniformities of this form
\emph{congruential uniformities}.
\end{example}

\begin{example} Let $A$ be an algebra and $\mathcal S$ a compatible semiuniformity. Then $\Ug^{|A|}\mathcal S$ is a compatible uniformity.
\begin{proof}
Theorem~\ref{E:CHSpace} gives a formula for $\Ug\mathcal S$. Given a sequence $\{\,U_n\,\}_{n\in\mathbb N_+}$ with $U_n\in\mathcal S$ for all $n$, and a basic operation $\omega$, there exists a $V_n\in\mathcal S$ for each $n$ such that $\omega(V_N)\subseteq U_n$. We have
\[[\,V _n\,:\,n\in\mathbb N_+\,])\subseteq[\,U_n\,:\,n\in\mathbb N_+\,],\]
because for all $n$,
\[\omega\left(\bigcup_{\gamma\in\Gamma_n}V_{\gamma(1)}\circ\ldots\circ V_{\gamma(n)}\right)
=\bigcup_{\gamma\in\Gamma_n}\omega\left(V_{\gamma(1)}\right)\circ\ldots\circ\omega\left(V_{\gamma(n)}\right)
\subseteq
\bigcup_{\gamma\in\Gamma_n}U_{\gamma(1)}\circ\ldots\circ U_{\gamma(n)}.\]
\end{proof}
\end{example}

\begin{example}
If $\mathcal U\in\OpUnif A$, then $\bigcap\mathcal
U\in\Con A$. We may consider this as
a mapping from $\OpUnif A$ to $\OpUnif A$, where we map
$\mathcal U$ to $\Fg\{\,\bigcap\mathcal U\,\}$; more
generally, we can map $\mathcal U$ to the filter of
$\kappa$-fold intersections of relations in $\mathcal U$,
for $\kappa$ some given infinite cardinal. The result
will be a compatible uniformity $\mathcal V$ such that
$\mathcal V$ admits $\kappa$-fold intersections of its
elements. We say that $\mathcal V$ \emph{satisfies the
$\kappa$-fold intersection property}.
\end{example}

\subsection{Uniformities and
homomorphisms}

\begin{theorem} Let $A$, $B$ be algebras, and $f:B\to
A$ a homomorphism. We have
\begin{enumerate}
\item[(1)] If $\mathcal U$ is a compatible uniformity
(compatible semiuniformity) on
$A$, then
$f^{-1}(\mathcal U)$ is a compatible uniformity
(respectively, compatible semiuniformity) on
$B$.
\item[(2)] The mapping $\mathcal U\mapsto f^{-1}(\mathcal
U)$ preserves arbitrary meets.
\end{enumerate}
\end{theorem}

\begin{theorem} Let $A$ and $B$ be algebras, and
$f:A\to B$ a homomorphism. Then the mapping $\mathcal
U\mapsto \Ug(f(\mathcal U))$ preserves arbitrary joins.
\end{theorem}

\begin{remark} In \cite[Section 11]{r02}, there is an
incorrect statement about the procedure for finding the
colimit of a diagram $F:\mathcal D\to\mathbf V[\Unif]$, in
the category
$\mathbf V[\Unif]$ of algebra objects in the category $\Unif$ and satisfying the identities of $\mathbf V$. The uniformity of the colimit
is the smallest
\emph{compatible} uniformity greater than or equal to all
of the $\iota_d(\mathcal U(d))$, where $\iota_d$ is the
insertion of $F(d)$ into the colimit and $\mathcal U(d)$
is the uniformity on $F(d)$.
\end{remark}

\subsection{Compatible uniformities on algebras in
congruence-permutable algebras}

We recall (\cite[Theorems~6.4 and 6.2]{r02}) that if $A$
is an algebra
in a congruence-permutable variety, and $\mathcal U$,
$\mathcal V\in\OpUnif A$, then $\mathcal U$ and $\mathcal
V$ permute, and
that
$\OpUnif A$ is modular.

\subsection{Compatible uniformities on algebras in congruence-modular varieties}

Contrary to what one might (and indeed, we did) conjecture from Day's Theorem\cite{Day}, and the success in generalizing Mal'tsev's Theorem to congruence-permutable algebras with a compatible uniformity, an algebra $A$ in a congruence-modular variety can have a lattice $\OpUnif A$ of compatible uniformities that is not modular.  Indeed:

\begin{example}[{\cite[Example 2.4]{weber}}]\label{E:UnifNonMod}
Let $C$ be an infinite chain, i.e., an infinite lattice which is totally ordered. Then $\OpUnif C$ is not modular. For details, we refer the reader to the discussion of this example in \cite{weber}. Note that the author gives lattices of uniformities the reverse ordering of the one we use, thus, interchanging the meet (inf) and join (sup) operations.
\end{example}

\section{Superequivalences}
\label{S:SuperEquivalences}

\begin{notation} If $\mathcal I\in\Idl\Rel S$ for some set $S$, then $\mathcal I^{\mathrm{op}}$ will denote the ideal $\{\,\mathcal R^{\mathrm{op}}\mid R\in\mathcal I\,\}$. If $\mathcal I$, $\mathcal I'\in\Idl\Rel S$, then $\mathcal I\circ\mathcal I'$ will denote the ideal $\Ig\{\,R\circ R'\mid R\in\mathcal I,R'\in\mathcal I'\,\}$.
\end{notation}

\begin{definition}
Consider conditions on $\mathcal I\in\Idl\Rel S$:
\begin{enumerate}
\item[$\mathrm{(SE_r)}$] $\Delta_S\in\mathcal I$;
\item[$\mathrm{(SE_s)}$] $\mathcal I=\mathcal I^{\mathrm{op}}$;
\item[$\mathrm{(SE_t)}$] $\mathcal I\circ\mathcal I\leq\mathcal I$,
\end{enumerate}
We say that $\mathcal I$ is \emph{reflexive} if it satisfies condition $\mathrm{(SE_r)}$, \emph{symmetric} if it satisfies condition $\mathrm{(SE_s)}$, \emph{transitive} if it satisfies condition $\mathrm{(SE_t)}$, a \emph{semisuperequivalence} if it satisfies $\mathrm{(SE_r)}$ and $\mathrm{(SE_s)}$, and a \emph{superequivalence} if it satisfies $\mathrm{(SE_r)}$, 
$\mathrm{(SE_s)}$, and
$\mathrm{(SE_t)}$.
\end{definition}

\begin{theorem}
Let $\mathcal I$ be an ideal of relations on a set $S$. The following are equivalent:
\begin{enumerate}
\item $\mathcal I$ is a superequivalence;
\item $\mathcal I$ is generated, as an ideal, by symmetric relations, and closed under composition of relations;
\item $\mathcal I$ is generated, as an ideal, by symmetric relations, and closed under relational powers.
\end{enumerate}
\end{theorem}

If $f:S\to T$ is a function from $S$ to another set $T$, then there is an inverse image mapping along $f$ from $\Idl\Rel T$ to $\Idl\Rel S$, given by
\[\mathcal I\mapsto f^{-1}(\mathcal I)\overset{\mathrm{def}}=\Ig\{\,f^{-1}(R)\mid R\in\mathcal I\,\},\]
and a direct image mapping along $f$ from $\OpSuperEqv S$ to $\OpSuperEqv T$, given by
\[\mathcal I\mapsto f(\mathcal I)\overset{\textrm{def}}=\{\,f(R)\mid R\in\mathcal I\,\};\]
and we have $\pair{[\mathcal I\mapsto f(\mathcal I)]}{[I'\mapsto f^{-1}(\mathcal I')]}:\Idl\Rel S\rightharpoonup\Idl\Rel T$.

\begin{definition} If $S$ is a set and $\mathcal I$ is a superequivalence (semisuperequivalence) then we say that $\pair S{\mathcal I}$ is a \emph{superequivalence space} (respectively, \emph{semisuperequivalence space}).
\end{definition}

\begin{definition}
If $\pair S{\mathcal I}$ and $\pair T{\mathcal I'}$ are superequivalence (semisuperequivalence) spaces, a function $f:S\to T$ is a \emph{super\-equivalent function} if the equivalent conditions $f(\mathcal I)\leq\mathcal I'$, $\mathcal I\leq f^{-1}(\mathcal I')$ are satisfied, which happens iff $R\in\mathcal I\implies f(R)\in\mathcal I'$.
\end{definition}

Superequivalence spaces, and the superequivalent functions between them, form a category, which we denote by $\SuperEqv$. Similarly, we have the category $\SemiSuperEqv$ of semisuperequivalence spaces.

The category-theoretic product of superequivalence or semisuperequivalence spaces $\pair S{\mathcal I}$ and $\pair T{\mathcal I'}$ is
\[\pair S{\mathcal I}\times\pair T{\mathcal I'}=
\pair{S\times T}{\Ig\{\,R\times\mathcal R'\mid\mathcal R\in\mathcal I,
\mathcal R'\in\mathcal I'\,\}}.\]

\begin{example} Let $S$ be a set, and $\alpha$ an equivalence relation on $S$. Then $\pair S{\Ig\{\,\alpha\,\}}$ is a superequivalence space.
\end{example}

\begin{example} Let $S$ be a set, and $\alpha$ a symmetric relation on $S$. Then $\pair S{\Ig^\circ\{\,\alpha\,\}}$ is a superequivalence space, where $\Ig^\circ\mathcal R$, for a set of relations $\mathcal R$, is the ideal generated by the set of finite compositions of relations in $\mathcal R$.
\end{example}

\begin{example}
\label{E:Xi}
Let $S$ be a set, and let $\mathcal R$ be the set of reflexive, symmetric relations on $S$ that relate only a finite number of off-diagonal pairs $s\neq s'$. Then $\pair S{\Ig^\circ\mathcal R}$ is a superequivalence space, which we denote by $\Xi(S)$.
\end{example}

[add = $\SEg\bigvee^{\Idl\Rel S}\bigcup_i\mathcal I_i$ - make def of $\SEg$]

\subsection{The lattices $\OpSuperEqv S$ and $\OpSemiSuperEqv S$}

We denote the set of superequivalences on a set $S$ by $\OpSuperEqv S$ and the set of semisuperequivalences by $\OpSemiSuperEqv$. They admit complete meet operations given by intersection of ideals, and complete join operations given by 
\[\bigvee_i^{\OpSuperEqv}\mathcal I_i=\Ig^\circ\{\,\bigcup_i\mathcal I_i\,\},\]
and
\[\bigvee_i^{\OpSemiSuperEqv}\mathcal I_i=\Ig\{\,\bigcup_i\mathcal I_i\,\}.\]

\subsection{$\SEg\mathcal I$, the superequivalence generated by an ideal $\mathcal I$}  Since $\SuperEqv S$ is a subset of $\Idl\Rel S$ closed under arbitrary meets, we can define $\SEg\mathcal I$, for any ideal of relations $\mathcal I$, as $\bigwedge_{\mathcal I\leq\mathcal J\in\SuperEqv S}\mathcal J$, and we then have $\pair{[\mathcal I\mapsto\SEg\mathcal I]}{\mathcal J\mapsto\mathcal J]}:\Idl\Rel S\rightharpoonup\SuperEqv S$.

\subsection{Permutability for superequivalences}

\begin{definition}
We say that two superequivalences $\mathcal I$, $\mathcal I'$ on a set $S$ \emph{permute} if
$\mathcal I\circ\mathcal I'=\mathcal I'\circ\mathcal I$.
\end{definition}

\begin{theorem}
Superequivalences $\mathcal I$ and $\mathcal I'$ permute iff $\mathcal I\vee\mathcal I'=\mathcal I\circ\mathcal I'$.
\end{theorem}

\begin{proof}
We have $\mathcal I\circ\mathcal I'\leq\mathcal I\vee\mathcal I'$ for any superequivalences $\mathcal I$ and $\mathcal I'$,
because $\mathcal I\vee\mathcal I'$ contains $\mathcal I$ and $\mathcal I'$ and satisfies $\mathrm{(SU_t)}$.
If $\mathcal I$ and $\mathcal I'$ permute, then $\mathcal I\circ\mathcal I'$ is a superequivalence, being an ideal of relations closed under composition, and so, it is the join. For, $\Ig\{\,\Delta\,\}\leq\mathcal I$, which implies $\mathcal I\circ\mathcal I'\leq\mathcal I'$, and similarly, $\mathcal I\circ\mathcal I'\leq\mathcal I$.

Conversely, if $\mathcal I\vee\mathcal I'=\mathcal I\circ\mathcal I'$, then 
we have
\begin{align*}
\mathcal I\circ\mathcal I'
&=\mathcal I\vee\mathcal I'\\
&=(\mathcal I\vee\mathcal I')^{\mathrm{op}}\\
&=(\mathcal I\circ\mathcal I')^{\mathrm{op}}\\
&=(\mathcal I')^{\mathrm{op}}\circ(\mathcal I)^{\mathrm{op}}\\
&=\mathcal I'\circ\mathcal I.
\end{align*}
\end{proof}

\begin{theorem}
Let $L$ be a sublattice of $\OpSuperEqv S$ for some set $S$. If the elements
of $L$ permute pairwise, then $L$ is modular.
\end{theorem}

\begin{proof}
 It suffices to show, given $\mathcal I$, $\mathcal I'$, $\mathcal I''\in\OpSuperUnif A$ such that $\mathcal I\leq\mathcal I''$, that
\[(\mathcal I\vee\mathcal I')\wedge\mathcal I''\leq\mathcal I\vee(\mathcal I'\wedge\mathcal I''),\]
as the opposite inequality holds in any lattice. The left-hand side is generated by relations $(R\circ R')\cap R''$ such that $R\in\mathcal I$, $R'\in\mathcal I'$, and $R''\in\mathcal I''$, so let $(R\circ R')\cap\mathcal R''$ be such a relation. Consider $R\circ((R'\wedge(R^{-1}\circ R''))\in\mathcal I\circ(\mathcal I'\wedge\mathcal I'')$.  If $a\mathrel{R}b\mathrel{R'}c$ and $a\mathrel{R''}c$, then 
$a\mathrel Rb$ and $b\mathrel{R'\cap(R^{-1}\circ R''))}c$, showing that $(R\circ R')\wedge R''\leq R\circ(R'\wedge(R^{-1}\circ R''))$.
\end{proof}

\begin{lemma}
\label{T:InversePi} If $\pi:A\to B$ is an onto function, and $\{\,\mathcal I_i\,\}_{i\in I}$ is a tuple of superequivalences on $B$, then
\[\pi^{-1}(\bigvee_i\mathcal I_i)=\bigvee_i\pi^{-1}(\mathcal I_i).\]
\end{lemma}

\begin{proof}
($\geq:$) By monotonicity of inverse image.

($\leq:$) Let $R\in\pi^{-1}\Ig^\circ\{\,\bigcup_i\mathcal I_i\,\}$. Then for some finite tuple $\iota_1$, $\ldots$, $\iota_n$ where each $\iota_j\in I$, we have
\begin{align*}
R&\leq\pi^{-1}(R_1\circ\ldots\circ R_n)\textrm{ where each }R_j\in\mathcal I_{\iota_j}\\
&=\pi^{-1}(R_1)\circ\ldots\circ\pi^{-1}(R_n)\\
&\in\Ig^\circ\{\,\bigcup_i\pi^{-1}(\mathcal I_i)\,\}\\
&=\bigvee_i\pi^{-1}(\mathcal I_i).
\end{align*}
Thus, $R\in\bigvee_i\pi^{-1}(\mathcal I_i)$.
\end{proof}

\begin{theorem} For any set $S$, $\OpSuperEqv S$ and $\OpSemiSuperEqv S$ are algebraic lattices.
\end{theorem}

\begin{remark} Algebraicity is a very important property. For lack of algebaicity of $\OpSemiUnif A$ for an algebra $A$ in a congruence-modular variety, we were only able to give a partial proof of modularity of the lattice $\OpUnif A$ in \cite{r02}. (Indeed, see Example~\ref{E:UnifNonMod}.)
\end{remark}

Algebraicity of $\OpSuperEqv S$ plays a crucial role in the proof of the following theorem:

\begin{theorem}\label{T:SECC} The category $\SuperEqv$ is cartesian closed.\end{theorem}

\begin{proof} Let the cartesian closedness adjunction
for $\Set$, parameterized by the set $b$, be
\[\alpha^b:\Set(a\times b,c)\cong\Set(a,c^b)\]
for sets $a$ and $c$.
For every
$\pair b{\mathcal I'}$, $\pair c{\mathcal I''}$ in
$\SuperEqv$, we define
$\pair c{\mathcal I''}^{\pair b{\mathcal
I'}}=\pair{c^b}{(\mathcal
I'')^{\mathcal I'}}$
where
\begin{align*}(\mathcal
I'')^{\mathcal I'}=&\SEg\left(\bigvee^{\Idl\Rel c^b}_{f\in\SuperEqv(\pair a{\mathcal I}\times\pair b{\mathcal I'},\pair
c{\mathcal I''})}\alpha^b(f)(\mathcal I)\right),\\
=&\left(\bigvee^{\Idl\Rel c^b}_{f\in\SuperEqv(\pair a{\mathcal I}\times\pair b{\mathcal I'},\pair
c{\mathcal I''})}\SEg\left(\alpha^b(f)(\mathcal I)\right)\right)
\end{align*}
and the adjunction $\alpha^{\pair b{\mathcal I'}}$ by
\[\alpha^{\pair b{\mathcal I'}}(f)=\alpha^b(f).\]
 
Our first task is to show that our formula for the function space
is functorial in $\pair c{\mathcal I''}$. This means
that if
$g:\pair c{\mathcal I''}\to
\pair{c'}{\mathcal I'''}$, we want $g^b\in\SuperEqv(\pair c{\mathcal I''}^{\pair b{\mathcal
I'}},\pair{c'}{\mathcal I'''}^{\pair b{\mathcal
I'}})$, or in other words, that
\[g^b((\mathcal I'')^{\mathcal I'})=g^
b\left(\bigvee^{\Idl\Rel c^b}_{f\in\SuperEqv(\pair a{\mathcal I}\times\pair b{\mathcal I'},\pair
c{\mathcal I''})}\alpha^b(f)(\mathcal I)\right)
\leq(\mathcal
I''')^{\mathcal I'}.\] However, forward image
preserves joins. Thus it suffices to prove that if
$f:\pair a{\mathcal I}\times\pair b{\mathcal
I'}\to\pair c{\mathcal I''}$, then
$g^b(\alpha^b(f)(\mathcal I))\leq\alpha^b(g\circ
f)(\mathcal I)$. This is true because by the
naturality of $\alpha^b$,
$g^b\circ\alpha^b(f)=\alpha^b(g\circ f)$.

The unit natural transformation
$\eta^{\pair b{\mathcal I'}}_{\pair a{\mathcal
I}}:\pair a{\mathcal I}\to(\pair a{\mathcal
I}\times\pair b{\mathcal I'})^{\pair b{\mathcal I'}}$
is just the arrow $\eta^b_a:a\to(a\times b)^b$, i.e., the unit natural transformation for the underlying adjunction in $\Set$, and is
an element of $\SuperEqv(\pair a{\mathcal
I},(\pair a{\mathcal I}\times\pair b{\mathcal
I'})^{\pair b{\mathcal I'}})$ because
$\eta^b_a=\alpha^b(1_{a\times b})=\alpha^b(1_{\pair
a{\mathcal I}\times\pair b{\mathcal I'}})$.

The counit natural transformation $\varepsilon^{\pair
b{\mathcal I'}}_{\pair c{\mathcal I''}}:\pair
c{\mathcal I''}^{\pair b{\mathcal I'}}\times\pair
b{\mathcal I'}\to\pair c{\mathcal I''}$ is the underlying counit arrow
$\varepsilon^b_c:c^b\times b\to c$, and is an element
of $\SuperEqv(\pair c{\mathcal I''}^{\pair
b{\mathcal I'}}\times\pair b{\mathcal I'},\pair
c{\mathcal I''})$. For, if
 we take the product of the
arrow $\alpha^b(f)$ and $1_b$, and then compose the
resulting element of $\Set(a\times b,c^b\times
b)$ with $\varepsilon^b_c$, we get back $f$, showing
that for any single $f\in\SuperEqv(\pair
a{\mathcal I}^{\pair b{\mathcal I'}},\pair c{\mathcal
I''})$, $\varepsilon^b_c(\Eqv(\alpha^b(f)(\mathcal
I)))\leq\Eqv(\varepsilon^b_c(\alpha^b(f)(\mathcal
I''))=\Eqv(f(\mathcal I))\leq\mathcal I''$.

The set of $\alpha^b(f)(\mathcal I)$ is directed, because given
$f:\pair a{\mathcal I}\times\pair b{\mathcal
I'}\to\pair c{\mathcal I''}$ and $\hat
f:\pair {\hat a}{\hat{\mathcal I}}\times\pair
b{\mathcal I'}\to\pair c{\mathcal I''}$, we have
\[\alpha^b(f)\cup\alpha^b(\hat f)\in\OpSuperEqv(
\pair a{\mathcal I}{\textstyle\coprod}\pair{\hat
a}{\hat{\mathcal I}},\pair{c^b}{(\mathcal
I'')^{\mathcal I'}}),\]
where $\coprod$ denotes the coproduct.
If $\iota_{\pair a{\mathcal I}}$ and
$\iota_{\pair{\hat a}{\hat{\mathcal I}}}$ are the
insertions of $\pair a{\mathcal I}$ and $\pair{\hat
a}{\hat{\mathcal I}}$ into the coproduct, then the
superequivalence $\mathcal I\coprod\hat{\mathcal
I}$ of the coproduct is $\SEg\left(\iota_{\pair a{\mathcal
I}(\mathcal I)}\vee\iota_{\pair{\hat a}{\hat{\mathcal I}}}(\hat{\mathcal I})\right)$, and
we have 
\[\alpha^b(f)(\mathcal I)\vee\alpha^b(\hat
f)(\hat{\mathcal I})\leq(\alpha^b(f)\cup\alpha^b(\hat
f))(\mathcal I{\textstyle \coprod}\hat{\mathcal I}).\]
Since $\SEg$ is a monotone operation, the set of $\SEg(\alpha^b(f)(\mathcal I)$ is also directed.
From
Lemma~\ref{T:InversePi} and the fact that
$\Idl\Rel(c^b\times b)$ is algebraic and thus
meet-continuous, we have
\begin{align*}(\mathcal
I'')^{\mathcal I'}\times\mathcal
I'
&=\pi^{-1}\left(\bigvee_f\SEg\left(\alpha^b(f)(\mathcal
I)\right)\right)\wedge(\pi')^{-1}(\mathcal I')\\
&=\left(\bigvee_f\pi^{-1}(\alpha^b(f)(\mathcal
I))\right)\wedge(\pi')^{-1}(\mathcal I')\\
&=\bigvee_f\left(\pi^{-1}\left(\SEg\left(\alpha^b(f)(\mathcal
I)\right)\right)\wedge(\pi')^{-1}(\mathcal I')\right);
\end{align*}
and it follows that $\varepsilon^{\pair b{\mathcal
I'}}_{\pair c{\mathcal I''}}$ is a superequivalent function, as stated.

The unit $\eta^{\pair b{\mathcal
I'}}$ and counit $\varepsilon^{\pair
b{\mathcal I'}}$ inherit
naturality and the equations
\[\left(\varepsilon^{\pair b{\mathcal I'}}_{\pair
a{\mathcal I}}\right)^{\pair b{\mathcal
I'}}\circ\eta^{\pair b{\mathcal I'}}_{\pair a{\mathcal
I}^{\pair b{\mathcal I'}}}=1_{\pair a{\mathcal
I}^{\pair b{\mathcal I'}}},\qquad
\varepsilon^{\pair b{\mathcal I'}}_{\pair a{\mathcal
I}\times\pair b{\mathcal I'}}\circ\left(\eta^{\pair
b{\mathcal I'}}_{\pair a{\mathcal I}}\times\pair
b{\mathcal I'}\right)=1_{\pair a{\mathcal I}\times\pair
b{\mathcal I'}}
\] (\cite[Equations IV.1(8)]{macl}), sufficient to
demonstrate the adjunction $\alpha^{\pair b{\mathcal
I'}}$ by
\cite[Theorem IV.1.2]{macl}, from the corresponding equations for the unit $\eta^b$
and counit
$\varepsilon^b$ of
$\alpha^b$.
\end{proof}

\subsection{Compatible superequivalences}

\begin{definition}
A superequivalence $\mathcal I$ on an algebra $A$ is \emph{compatible} if for every $n$-ary basic operation $\omega$, for every $n$, we have
$\omega(R_1,\ldots,R_n)\overset{\textrm{def}}=\{\,\pair{\omega(a_1,\ldots,a_n)}{\omega(b_1,\ldots,b_n)}\mid \pair{a_i}{b_i}\in R_i\text{ for all $i$}\,\}\in\mathcal I$ for all $R_1$, $\ldots$, $R_n\in\mathcal I$.
\end{definition}

\begin{remark} In other words, a superequivalence $\mathcal I$ on $A$ is compatible with an $n$-ary operation $w$ iff $w$ is a superequivalent function from $\pair A{\mathcal I}^n$ to $\pair A{\mathcal I}$.
\end{remark}

\begin{notation}
We denote the set of compatible superequivalences (semisuperequivalences) on an algebra $A$ by $\OpSuperEqv A$ (respectively, $\OpSemiSuperEqv A$).
\end{notation}

\begin{notation}[Continuation of Example~\ref{E:Xi}]
If $R$ is a ring, then $\Xi(A)\notin\SuperEqv A$.
\end{notation}

\begin{example}\label{E:SEFgCongruences}
If $A$ is an algebra, the set of congruences on $A$ which are finitely generated (i.e., which as equivalence relations, are generated by a finite number of pairs of elements) is a directed set, and thus a ceiling of a compatible superequivalence on $A$.
\begin{proof}
The set of such congruences is directed because if $S_1$ and $S_2$ are finite sets of pairs, then \[\Cg S_1\circ\Cg S_2\subseteq Cg S_1\vee\Cg S_2\subseteq \Cg(S_1\cup S_2).\]
\end{proof}
\end{example}

\begin{theorem} The meet and join of tuples of superequivalences compatible with an operation $w$ are compatible with $w$.
\end{theorem}

\begin{proof} Let $\mathcal I_i$ be superequivalences on $A$ compatible with an $n$-ary operation $w$. We have $\bigwedge_i\mathcal I_i=\bigcap_i\mathcal I_i$. We must show that if $R_1$, $\ldots$, $R_n\in\mathcal I_i$ for all $i$, then $w(R_1,\ldots,R_n)\in\mathcal I_i$ for all $i$. However, each $\mathcal I_i$ is compatible with $w$.

On the other hand, we have $\bigvee_i\mathcal I_i=\Ig^\circ\{\,\bigcup_i\mathcal I_i\,\}$. Given $R_1$, $\ldots$, $R_n\in\bigvee_i\mathcal I_i$, we need to have $w(R_1,\ldots,R_n)\in\bigvee_i\mathcal I_i$. However, each $R_j$ is contained in a finite composition of relations $R'_{j,k}$ for some relations $R'_{j,k}\in\mathcal I_{k}$, and
$w(R_1,\ldots,R_n)$ is thus a finite composition of relations of the form
$w(\Delta,\ldots,\Delta,R'_{j,k},\Delta,\ldots,\Delta)$, which is a relation in $\mathcal I_{k}$ because $\mathcal I_{k}$ is compatible with $w$.
Thus, $w(R_1,\ldots,R_n)$ is contained in a finite composition of elements of $\bigcup_i\mathcal I_i$.
\end{proof}

\subsection{$\SEg^A\mathcal I$ and $\SEg^{|A|}\mathcal I$} As a corollary, given an ideal of relations $\mathcal I$ on $A$, we can define
\[\SEg^A\mathcal I=\bigwedge_{\mathcal I\leq\mathcal J\in\SuperEqv A}\mathcal J,\]
and we have $\pair{[\mathcal I\mapsto\SEg^A\mathcal I]}{[\mathcal J\mapsto \mathcal J]}:\Idl\Rel A\rightharpoonup\OpSuperEqv A$.

As with $\Ug$, we differentiate between $\SEg^A\mathcal I$ and $\SEg^{|A|}\mathcal I$.

\begin{lemma}  If $\mathcal I$ is a symmetric ideal of relations on $A$, then
\[\SEg^{|A|}\mathcal I=\Ig^\circ\{\,\bigcup_i\mathcal I_i\,\}.\]
\end{lemma}

\subsection{Compatible superequivalences of congruence-permutable algebras}

As with $\Unif$, Mal'tsev's Theorem generalizes easily:

\begin{theorem}
Let $A$ be an algebra with a Mal'tsev term $p$. Then
any two compatible superequivalences on $A$ permute.
\end{theorem}

\begin{proof}
Let $\mathcal I$, $\mathcal I'$ be compatible superequivalences on $A$.
It suffices to show $\mathcal I\circ\mathcal I'\leq\mathcal I'\circ\mathcal I$. Thus it suffices to show, if $R\in\mathcal I$, $R'\in\mathcal I'$, then $R\circ R'\in\mathcal I'\circ\mathcal I$.

$\Delta$ belongs to every superequivalence, by $\mathrm{(SE_r)}$.
Thus we have
\[p(\Delta,\Delta,R')\circ p(R,\Delta,\Delta)\in\mathcal I'\circ\mathcal I\]
since $\mathcal I$ and $\mathcal I'$ are compatible with the term operation $p$.
$a\mathrel R b\mathrel R' c$ then implies
\[a=p(a,b,b)\mathrel {p(\Delta,\Delta,R')} p(a,b,c)
\mathrel{p(R,\Delta,\Delta)}p(b,b,c)=c,\]
proving that
$R\circ R'\leq  p(\Delta,\Delta,R')\circ
p(R,\Delta,\Delta)\in\mathcal I'\circ\mathcal I$.
\end{proof}

\begin{corollary}
Let $A$ be an algebra with a Mal'tsev term.  Then $\OpSuperEqv A$ is modular.
\end{corollary}

\begin{example}
Let $A$ be an algebra with a Mal'tsev term. Then an ideal in $\Con A$ is a compatible superequivalence, because if $\alpha$, $\beta\in\mathcal I$, then $\alpha\circ\beta=\alpha\vee\beta\in\mathcal I$.
\end{example}

\subsection{Compatible superequivalences of algebras in congruence-modular
varieties}\label{S:CMSE}

The above results show  that existence of a Mal'tsev term, in a variety containing an
algebra $A$, influences the structure of the lattice
$\OpSuperEqv A$ in the same way it influences that of $\Con
A$. 
We are led to ask whether (unlike the situation for algebra objects in $\Unif$)  the existence of a sequence of Day terms in a
variety $\mathbf V$ forces $\OpSuperEqv A$ to be modular when
$A\in\mathbf V$.

\begin{notation} In connection with Day terms, given a relation $R$ on an algebra $A$, we will define
\begin{align*}
\mathbf m(R)&=\bigcup_im_i(R,R,R,R)\\
&=\bigcup_i\{\,\pair{m_i(a,b,c,d)}{m_i(a',b',c',d')}\mid
a\mathrel Ra',b\mathrel Rb',c\mathrel Rc',\text{ and }d\mathrel Rd'\,\}.
\end{align*}\end{notation}

\begin{lemma}\label{T:mofRSE} We have
\begin{enumerate}
\item
$R\subseteq\mathbf m(R)$.
\item If a superequivalence $\mathcal I$ is compatible with the operations $m_i$, and $R\in\mathcal I$, then $\mathbf m(R)\in\mathcal I$.
\end{enumerate} 
\end{lemma}

\begin{proof}
(1):
$m_i(a,a,a,a)=a$ for all $i$ and all $a\in A$.

(2): The tuple $\mathbf m$ is finite. Since $\mathcal I$ is compatible, $m_i(R)\in\mathcal I$ for each $i$, but $\mathcal I$ is an ideal of relations.

\end{proof}

For the following discussion, $\mathbf V$ will be a
congruence-modular variety, $m_i$, $i=0$, $\ldots$, $\D$
will be a sequence of Day terms for $\mathbf V$, and $A$
will be an algebra in $\mathbf V$. Our first goal will be
to prove a generalization of the Shifting Lemma
\cite{Gumm}.

\begin{lemma} \label{T:Lemma1Eqv}
Let $\mathcal I\in\OpSuperEqv A$. If $R\in\mathcal I$, and $Y=\mathbf m(R')^{2\D}$, the relational composition of $2\D$ copies of the relation $\mathbf m(R')$,  where $R'=R\cup R^{\mathrm{op}}$, then if 
$a$, $b$, $c$, $d\in A$ with $b\mathrel{R}d$ and
$m_i(a,a,c,c)\mathrel{R}m_i(a,b,d,c)$ for all $i$,
we have
$a\mathrel Yc$.
\end{lemma}

\begin{proof}
Let $u_i=m_i(a,b,d,c)$ and $v_i=m_i(a,a,c,c)$ for all $i$.
If $b\mathrel Rd$ and $m_i(a,a,c,c)\mathrel Rm_i(a,b,d,c)$ for all $i$, then for even $i<\D$, we have
$u_i\mathrel{\mathbf m(R')}v_i=v_{i+1}\mathrel{\mathbf m(R')}u_{i+1}$, while for  odd $i<\D$, we have
$u_i\mathrel{\mathbf m(R')}m_i(a,b,b,c)=m_{i+1}(a,b,b,c)\mathrel{\mathbf m(R')}u_{i+1}$.

Thus, for $i<\D$, $u_i\mathrel{\mathbf m(R')\circ\mathbf m(R')}u_{i+1}$, and consequently, $a=u_0\mathrel Yu_\D=c$.
\end{proof}

\begin{lemma}[A Shifting Lemma for Superequivalences] \label{T:ShiftSuperEqv}
Let $\mathcal R\in\Idl\Rel A$ satisfy $\mathrm{(SE_r)}$ and $\mathrm{(SE_s)}$ and be compatible, and let
$\mathcal I_1$,
$\mathcal I_2\in\OpSuperEqv A$ be such that
$\mathcal R\wedge\mathcal I_1\leq\mathcal I_2\leq\mathcal I_1$. Then
$(\mathcal R\circ(\mathcal I_1\wedge\mathcal I_2)\circ\mathcal
R)\wedge\mathcal I_1\leq\mathcal I_2$.
\end{lemma}

\begin{proof} 
It suffices to show that, given relations $R\in\mathcal R$, $F\in\mathcal I_1$, and $X\in\mathcal I_1\wedge\mathcal I_2=\mathcal I_1\cap\mathcal I_2$, there is a relation $Y\in\mathcal I_2$ such that
\[(R\circ X\circ R)\cap F\subseteq Y;\]
to prove this, it suffices to construct $Y$ and show that $a\mathrel Rb\mathrel Xd\mathrel Rc$ and $a\mathrel Fc$ imply $a\mathrel Yc$.

Without loss of generality, we may assume that the given $R$, $F$, and $X$ are symmetric, because $\mathcal R$, $\mathcal I_1$, and $\mathcal I_2$ all satisfy $\mathrm{(SE_s)}$.

Let $W=\mathbf m(X)\cup(\mathbf m(R)\cap(\mathbf m(F)\circ\mathbf m(F)\circ\mathbf m(X)))$. Since $\mathcal R$ is compatible, $\mathcal I_1$ and $\mathcal I_2$ are compatible superequivalences, and $\mathcal R\cap\mathcal I_1\subseteq\mathcal I_2$, we have $W\in\mathcal I_2$.

If $a\mathrel Rb\mathrel Xd\mathrel Rc$ and $a\mathrel Fc$, then
$m_i(a,a,c,c)\mathrel{\mathbf m(R)}m_i(a,b,d,c)$ for all $i$, and $m_i(a,a,c,c)\mathrel{\mathbf m(F)}m_i(a,a,a,a)=a=m_i(a,b,b,a)\mathrel{\mathbf m(F)}m_i(a,b,b,c)\mathrel{\mathbf m(X)}m_i(a,b,d,c)$ for all $i$. Thus, $m_i(a,a,c,c)\mathrel Wm_i(a,b,d,c)$ for all $i$. We also have $b\mathrel Wd$ because $b\mathrel Xd$.
Then by Lemma~\ref{T:Lemma1Eqv}, $a\mathrel Yc$ where $Y=(W\cup W^{\mathrm{op}})^{2\D}$.
\end{proof}

Finally, we have

\begin{theorem}\label{T:ModularSE} The lattice $\OpSuperEqv A$ is modular.
\end{theorem}

\begin{proof}
Let $\mathcal I$, $\mathcal I'$,
$\mathcal I''\in\OpSuperUnif A$ be such
that $\mathcal I\leq\mathcal I''$. Then
\[\mathcal I\vee(\mathcal I'\wedge\mathcal I'')
\leq(\mathcal I\vee\mathcal I')\wedge\mathcal I'',
\]
as this inequality holds in any
lattice when $\mathcal I\leq\mathcal I''$.

Define $\mathcal R_k$ and $\mathcal R'_k$ for every natural number $k$ as $\mathcal
R_0=\mathcal R'_0=\mathcal I'$, $\mathcal R_{k+1}=\mathcal R_k\circ\mathcal I\circ\mathcal R_k$,
and $\mathcal R'_k=\mathcal R'_k\circ(\mathcal I_1\wedge\mathcal I_2)\circ\mathcal R'_k$ where $\mathcal I_1=\mathcal I''$ and $\mathcal I_2=\mathcal I\vee(\mathcal I'\wedge\mathcal I'')$. Since $\mathcal I\leq\mathcal I_2\leq\mathcal I_1$ we have for every $k$, $\mathcal R_k\leq\mathcal R'_k$. We also have $\mathcal R'_0\wedge\mathcal I_1=\mathcal I'\wedge\mathcal I''\leq\mathcal I_2$. Then by Lemma~\ref{T:ShiftSuperEqv} and induction on $k$, we have
$\mathcal R'_k\wedge\mathcal I''=\mathcal R'_k\wedge\mathcal I_1\leq\mathcal I_2$ for all $k$. Since the lattice $\Idl\Rel A$ is algebraic, and the sequence $\mathcal R'_k$ is increasing,
we have
\begin{align*}
(\mathcal I\vee\mathcal I')\wedge\mathcal I''
&=\left(\bigcup_k\mathcal R_k\right)\wedge\mathcal I''\\
&\leq\left(\bigcup_k\mathcal R'_k\right)\wedge\mathcal I''\\
&=\bigcup_k\left(\mathcal R'_k\wedge\mathcal I''\right)\\
&\leq\mathcal I_2\\
&=\mathcal I\vee(\mathcal I'\wedge\mathcal I'').
\end{align*}
\end{proof}

\section{Superuniformities}
\label{S:SuperUniformities}

\begin{notation} If $\mathcal E\in\Idl\Fil\Rel S$ for some set $S$, then $\mathcal E^{\mathrm{op}}$ will denote the ideal $\{\,\mathcal F^{\mathrm{op}}\mid\mathcal F\in\mathcal E\,\}$. If $\mathcal E$, $\mathcal E'\in\Idl\Fil\Rel S$, then $\mathcal E\circ\mathcal E'$ will denote the ideal $\Ig\{\mathcal F\circ\mathcal F'\mid\mathcal F\in\mathcal E,\mathcal F'\in\mathcal E'\,\}$.
\end{notation}

\begin{definition}
Consider the following conditions on $\mathcal E\in\Idl\Fil\Rel S$:
\begin{enumerate}
\item[$\mathrm{(SU_r)}$] $\Fg\{\,\Delta_S\,\}\in\mathcal E$;
\item[$\mathrm{(SU_s)}$] $\mathcal E=\mathcal E^{\mathrm{op}}$; and
\item[$\mathrm{(SU_t)}$] $\mathcal E\circ\mathcal E\leq\mathcal E$:
\end{enumerate}
We say that $\mathcal E$ is \emph{reflexive} if it satisfies condition $\mathrm{(SU_r)}$, \emph{symmetric} if it satisfies condition $\mathrm{(SU_s)}$, \emph{transitive} if it satisfies condition $\mathrm{(SU_t)}$, a \emph{semisuperuniformity} if it satisfies $\mathrm{(SU_r)}$ and $\mathrm{(SU_s)}$, and a \emph{superuniformity} if it satisfies $\mathrm{(SU_r)}$,
$\mathrm{(SU_s)}$, and
$\mathrm{(SU_t)}$.
\end{definition}

\begin{theorem}
Let $\mathcal E$ be an ideal of filters of relations on a set $S$. The following are equivalent:
\begin{enumerate}
\item $\mathcal E$ is a superuniformity;
\item $\mathcal E$ is generated, as an ideal, by semiuniformities, and closed under composition of filters of relations;
\item $\mathcal E$ is generated, as an ideal, by semiuniformities, and closed under relational powers of filters.
\end{enumerate}
\end{theorem}

If $f:S\to T$ is a function from $S$ to another set $T$, then there is an inverse image mapping from $\Idl\Fil\Rel T$ to $\Idl\Fil\Rel S$, given by
\[\mathcal E\mapsto f^{-1}(\mathcal E)\overset{\mathrm{def}}=\Ig\{\,f^{-1}(\mathcal F)\mid\mathcal F\in\mathcal E\,\},\]
and a direct image mapping from $\Idl\Fil\Rel S$ to $\Idl\Fil\Rel T$, given by
\[\mathcal E\mapsto f(\mathcal E)\overset{\textrm{def}}=\Ig\{\,f(\mathcal F)\mid\mathcal F\in\mathcal E\,\}.\]
We have $\pair{[\mathcal E\mapsto f(\mathcal E)]}{\mathcal E'\mapsto f^{-1}(\mathcal E')]}:\Idl\Fil\Rel S\rightharpoonup\Idl\Fil\Rel T$.

\begin{definition} If $S$ is a set and $\mathcal E$ is a superuniformity (semisuperuniformity) then we say that $\pair S{\mathcal E}$ is a \emph{superuniform space} (respectively, \emph{semisuperuniform space}).
\end{definition}

\begin{definition}
If $\pair S{\mathcal E}$ and $\pair T{\mathcal E'}$ are superuniform (or semisuperuniform) spaces, a function $f:S\to T$ is \emph{super\-uniform} if the equivalent conditions $\mathcal E\leq f^{-1}(\mathcal E')$, $f(\mathcal E)\leq\mathcal E'$ are satisfied.
\end{definition}

Superuniform spaces, and the superuniform functions between them, form a category in an obvious manner, which we denote by $\SuperUnif$.
The category-theoretic product of superuniform spaces $\pair S{\mathcal E}$ and $\pair T{\mathcal E'}$ is
\[\pair S{\mathcal E}\times\pair T{\mathcal E'}=
\pair{S\times T}{\Ig\{\,\mathcal F\times\mathcal F'\mid\mathcal F\in\mathcal E,
\mathcal F'\in\mathcal E'\,\}}.\]

\begin{example}\label{E:IgU} Let $\pair S{\mathcal U}$ be a uniform space. Then $\pair S{\Ig\{\,\mathcal U\,\}}$ is a superuniform space.
\begin{proof}
$\Ig\{\,\mathcal U\,\}$ is generated by a set of uniformities, namely $\{\,\mathcal U\,\}$, and is closed under composition of filters by $\mathrm{(U_t)}$.
\end{proof}
\end{example}

\begin{example}
\label{E:PhiOff}
Suppose $S$ be a set, $\pair T{\mathcal T}$ a hausdorff topological space, and $f:S\to T$ a function. For each compact subset $C\subseteq T$, and each entourage $U$ of the unique uniformity on $\pair C{\mathcal T|_C}$ giving rise to $\mathcal T|_C$, consider the relation \[\{\,\pair s{s'}\mid f(s)=f(s')\text{ or }f(s),f(t)\in C\text{ and }f(s)\mathrel Uf(t)\,\}.\] These relations form a filter $\mathcal S_C$ of relations on $S$, which is a semiuniformity. The ideal of filters on $S$ generated by the $\mathcal S_C$ is closed under composition of filters, and is thus a superuniformity, which we denote by $\Phi(f,\mathcal T)$.
\end{example}

\subsection{The lattices $\OpSuperUnif S$ and $\OpSemiSuperUnif S$}

We denote the set of superuniformities on a set $S$ by $\OpSuperUnif S$ and the set of semisuperuniformities by $\OpSemiSuperUnif$. They admit complete meet operations given by intersection of ideals, and complete join operations given by 
\[\bigvee_i^{\OpSuperUnif}\mathcal E_i=\Ig^\circ\{\,\bigcup_i\mathcal E_i\,\},\]
and
\[\bigvee_i^{\OpSemiSuperUnif}\mathcal E_i=\Ig\{\,\bigcup_i\mathcal E_i\,\},\]
where by $\Ig^\circ$ of a set  $X$ of filters we mean the ideal generated by the set of all finite compositions of filters in $X$. The lattice $\OpSemiSuperUnif S$ is distributive.

\subsection{Permutability for superuniformities}

\begin{definition}
We say that two superuniformities $\mathcal E$, $\mathcal E'$ on a set $S$ \emph{permute} if
$\mathcal E\circ\mathcal E'=\mathcal E'\circ\mathcal E$,
where $\mathcal E\circ\mathcal E'=\Ig\{\,\mathcal F\circ\mathcal F'\mid
\mathcal F\in\mathcal E\text{ and }\mathcal F'\in\mathcal E'\,\}$.
\end{definition}

\begin{theorem}
Superuniformities $\mathcal E$ and $\mathcal E'$ permute iff $\mathcal E\vee\mathcal E'=\mathcal E\circ\mathcal E'$.
\end{theorem}

\begin{proof}
We have $\mathcal E\circ\mathcal E'\leq\mathcal E\vee\mathcal E'$ for any superuniformities $\mathcal E$ and $\mathcal E'$, 
because $\mathcal E\vee\mathcal E'$ contains $\mathcal E$ and $\mathcal E'$ and satisfies $\mathrm{(SU_t)}$. Similarly, $\mathcal E'\circ\mathcal E\leq\mathcal E\vee\mathcal E'$.
On the other hand, if $\mathcal E\circ\mathcal E'$ is a superuniformity, we have $\mathcal E\vee\mathcal E'\leq\mathcal E\circ\mathcal E'$.  For, $\Ig\{\,\Fg\{\,\Delta\,\}\,\}\leq\mathcal E$ by $\mathrm{(SU_r)}$, which implies $\mathcal E'\leq\mathcal E\circ\mathcal E'$ by monadicity, and similarly, $\mathcal E\leq\mathcal E\circ\mathcal E'$.

Now, if $\mathcal E$ and $\mathcal E'$ permute, then $\mathcal E\circ\mathcal E'=\mathcal E'\circ\mathcal E$ is a superuniformity, being an ideal of filters closed under composition, and so, $\mathcal E\vee\mathcal E'\leq\mathcal E\circ\mathcal E'\leq\mathcal E\vee\mathcal E'$.

On the other hand, if $\mathcal E\vee\mathcal E'=\mathcal E\circ\mathcal E'$, then $\mathcal E\circ\mathcal E'$ is a superuniformity, and $\mathcal E'\circ\mathcal E\leq\mathcal E\vee\mathcal E'=\mathcal E\circ\mathcal E'$.
Since $\mathcal E'\circ\mathcal E\leq\mathcal E\circ\mathcal E'$, we have
\begin{align*}
\mathcal E\circ\mathcal E'
&=(\mathcal E'\circ\mathcal E)^{-1}\\
&\leq(\mathcal E\circ\mathcal E')^{-1}\\
&=\mathcal E'\circ\mathcal E.
\end{align*}

\end{proof}

\begin{theorem}
Let $L$ be a sublattice of $\OpSuperUnif S$ for some set $S$. If the elements
of $L$ permute pairwise, then $L$ is modular.
\end{theorem}

\begin{proof}
 It suffices to show, given $\mathcal E$, $\mathcal E'$, $\mathcal E''\in\OpSuperUnif A$ such that $\mathcal E\leq\mathcal E''$, that
\[(\mathcal E\vee\mathcal E')\wedge\mathcal E''\leq\mathcal E\vee(\mathcal E'\wedge\mathcal E''),\]
as the opposite inequality holds in any lattice. The left-hand side is generated by filters $(\mathcal F\circ\mathcal F')\cap\mathcal F''$ such that $\mathcal F\in\mathcal E$, $\mathcal F'\in\mathcal E'$, and $\mathcal F''\in\mathcal E''$ so let $(\mathcal F\circ\mathcal F')\cap\mathcal F''$ be such a filter. Consider $\mathcal F\circ((\mathcal F'\wedge(\mathcal F^{\textbf{op}}\circ\mathcal F''))\in\mathcal E\circ(\mathcal E'\wedge\mathcal E'')$. This filter is generated by relations
$F\circ(F'\cap({\tilde F}^{\textbf{op}}\circ F''))$ such that $F_1$, $\tilde F\in\mathcal F$, $F'\in\mathcal F'$, and $F''\in\mathcal F''$. Consider the filter
$((F\cap\tilde F)\circ F')\cap F''\in(\mathcal F\circ\mathcal F')\wedge\mathcal F''$. If $a\mathrel{F\cap\tilde F}b\mathrel{F'}c$ and $a\mathrel{F''}c$, then 
$a\mathrel Fb$ and $b\mathrel{F'\cap({\tilde F}^{\text{op}}\circ F''))}c$, showing that $(\mathcal F\circ\mathcal F')\wedge\mathcal F''\leq\mathcal F\circ(\mathcal F'\wedge(\mathcal F^{\text{op}}\circ\mathcal F''))$.
\end{proof}

\begin{theorem} For any set $S$, $\OpSuperUnif S$ and $\OpSemiSuperUnif S$ are algebraic lattices.
\end{theorem}

\begin{example}\label{E:Z} Let $\pair S{\mathcal I}$ be a superequivalence space. Then
\[Z(\mathcal I)\overset{\textrm{def}}=\Ig\{\,\Fg\{\,J\,\}\mid J\in\mathcal I\,\}\]
is a superuniformity on $S$. The mapping $\pair S{\mathcal I}\mapsto\pair S{Z(\mathcal I)}$ is a lattice homomorphism which preserves arbitrary joins.
\begin{proof}
By monotonicity of the mapping $\mathcal I\mapsto Z(\mathcal I)$, we have
$Z(\mathcal I\cap\mathcal I')\leq Z(\mathcal I)\wedge Z(\mathcal I')$. Let $\mathcal F\in Z(\mathcal I)\wedge Z(\mathcal I')$. Then $\mathcal F\leq\Fg\{\,J\,\}$ for some $J\in\mathcal I$ and $\mathcal F\leq\Fg\{\,J'\,\}$ for some $J'\in\mathcal I'$. That is, there is an $R\in\mathcal F$ such that $R\subseteq J\cap J'$. Then since $J\cap J'\in\mathcal I\cap\mathcal I'$, we have $\mathcal F\in Z(\mathcal I\cap\mathcal I')$.

Again by monotonicity, we have
$\bigvee_iZ(\mathcal I_i)\leq Z(\bigvee_i\mathcal I_i)$.
Suppose $\mathcal F\in Z(\bigvee_i\mathcal I_i)$. Then for some $J_1$, $\ldots$, $J_n\in\mathcal I_{i_j}$, $\mathcal F\leq\Fg\{\,K_1\circ\ldots\circ J_n\,\}$. However,
$\Fg\{\,J_1\circ\ldots\circ J_n\,\}=\Fg\{\,J_1\,\}\circ\ldots\circ\Fg\{\,J_n\,\}$. Thus, $\mathcal F\in\bigvee_iZ(\mathcal I_i)$.
\end{proof}
\end{example}

\begin{theorem}
\label{T:JoinIsoThm} If $\pi:A\to B$ is an onto function, and $\{\,\mathcal E_i\,\}_{i\in I}$ is a tuple of superuniformities on $B$, then
\[\pi^{-1}(\bigvee_i\mathcal E_i)=\bigvee_i\pi^{-1}(\mathcal E_i).\]
\end{theorem}

\begin{proof}
($\geq:$) By monotonicity of inverse image.

($\leq:$) Let $\mathcal F\in\pi^{-1}\Ig^\circ\{\,\bigcup_i\mathcal E_i\,\}$. Then for some finite tuple $\iota_1$, $\ldots$, $\iota_n$ where each $\iota_j\in I$, we have
\begin{align*}
\mathcal F&\leq\pi^{-1}(\mathcal F_1\circ\ldots\circ\mathcal F_n)\textrm{ where each }\mathcal F_j\in\mathcal E_{\iota_j}\\
&=\pi^{-1}(\mathcal F_1)\circ\ldots\circ\pi^{-1}(\mathcal F_n)\\
&\in\Ig^\circ\{\,\bigcup_i\pi^{-1}(\mathcal E_i)\,\}\\
&=\bigvee_i\pi^{-1}(\mathcal E_i).
\end{align*}
Thus, $\mathcal F\in\bigvee_i\pi^{-1}(\mathcal E_i)$.
\end{proof}

Algebraicity of $\OpSuperUnif S$ is again crucial for the proof of the following:

\begin{theorem}
 The category $\SuperUnif$ is cartesian closed.
\end{theorem}

\begin{proof} The proof is almost identical to the proof that $\SuperEqv$ is cartesian-closed (Theorem~\ref{T:SECC}).
\end{proof}

\subsection{Compatible superuniformities}

\begin{definition}
A superuniformity $\mathcal E$ on an algebra $A$ is \emph{compatible} if for every $n$-ary basic operation $w$, for every $n$, we have
$w(\mathcal F_1,\ldots,\mathcal F_n)\overset{\textrm{def}}=\Fg\{\,w(F_1,\ldots,F_n)\mid F_1\in\mathcal F_1,\ldots,F_n\in\mathcal F_n\,\}\in\mathcal E$ for all $\mathcal F_1$, $\ldots$, $\mathcal F_n\in\mathcal E$.
\end{definition}

\begin{remark} In other words, a superuniformity $\mathcal E$ on $A$ is compatible with an $n$-ary operation $w$ iff $w$ is superuniform as a function from $\pair A{\mathcal E}^n$ to $\pair A{\mathcal E}$.
\end{remark}

\begin{lemma}
\label{T:RelationsLemma}
If $w$ is an $n$-ary operation on a set $S$, and $\mathcal F_1$, $\ldots$, $\mathcal F_n$, $\mathcal F'$ are filters of relations on $S$, then for each $j$ such that $1\leq j\leq n$,
\begin{align*}w(\mathcal F_1,&\ldots,\mathcal F_{j-1},\mathcal F_j\circ\mathcal F',\mathcal F_{j+1},\ldots,\mathcal F_n)\\&\leq w(\mathcal F_1,\ldots,\mathcal F_n)\circ w(\Fg\{\,\Delta\,\},\ldots,\Fg\{\,\Delta\,\},\mathcal F',\Fg\{\,\Delta\,\},\ldots,\Fg\{\,\Delta\,\}).\end{align*}
\end{lemma}

\begin{proof} The left-hand side has a base of relations of the form
\[w(R_1,\ldots,R_{j-1},R_j\circ R',R_{j+1},\ldots,R_n),\]
and the right-hand side has a base of relations of the form
\[w(R_1,\ldots,R_n)\circ w(\Delta,\ldots,\Delta,F',\Delta,\ldots,\Delta),\]
where $F_i\in \mathcal F_j$ and $R'\in\mathcal R'$.
However,
\[w(R_1,\ldots,R_{j-1},R_j\circ R',R_{j+1},\ldots,R_n)\\
\subseteq w(R_1,\ldots,R_nn)\circ w(\Delta,\ldots,\Delta,R',\Delta\ldots,R_n).
\]
\end{proof}

\begin{theorem} The meet and join of tuples of superuniformities compatible with an operation $w$ are compatible with $w$.
\end{theorem}

\begin{proof} Let $\mathcal E_i$ be superuniformities on $A$ compatible with an $n$-ary operation $w$. We have $\bigwedge_i\mathcal E_i=\bigcap_i\mathcal E_i$. We must show that if $\mathcal F_1$, $\ldots$, $\mathcal F_n\in\mathcal E_i$ for all $i$, then $w(\mathcal F_1,\ldots,\mathcal F_n)\in\mathcal E_i$ for all $i$. However, each $\mathcal E_i$ is compatible.

On the other hand, we have $\bigvee_i\mathcal E_i=\Ig^\circ\{\,\bigcup_i\mathcal E_i\,\}$. Given $w$ and $\mathcal F_1$, $\ldots$, $\mathcal F_n\in\bigvee_i\mathcal E_i$, we need to have $w(\mathcal F_1,\ldots,\mathcal F_n)\in\bigvee_i\mathcal E_i$. However, from Lemma~\ref{T:RelationsLemma}, it follows that $w(\mathcal F_1,\ldots,\mathcal F_n)$ is $\leq$ a finite composition of elements of $\bigcup_i\mathcal E_i$.
\end{proof}

\begin{example}[Continuation of Example~\ref{E:IgU}] If $\mathcal U$ is a compatible uniformity on an algebra $A$, then $\Ig\{\,\mathcal U\,\}$ is a compatible superuniformity on $A$.
\end{example}

\begin{example}[Continuation of Example~\ref{E:PhiOff}] Let $A$ be an algebra, $\pair B{\mathcal T}$ be an algebra of the same type with a compatible topology, and $f:A\to B$ a homomorphism. Then $\Phi(f,\mathcal T)$ is not necessarily a compatible superuniformity on $A$. For, consider $A=\mathbb Q$, the ring of rational numbers, $B=\mathbb C$, the ring of complex numbers, $\mathcal T$ the usual topology on $\mathbb C$, and $f$ the embedding. Let $w$ be the basic operation of addition, $\mathcal F=f^{-1}(\{\,\pair xy\mid |x-y|\leq\varepsilon\,\})\cap\{\,U^2\cap\Delta_{\mathbb C}\,\})$, where $\varepsilon>0$ and $U$ is the unit ball. Then $\mathcal F+\mathcal F\notin\Phi(f,\mathcal T)$, because it relates $x$ and $x+\varepsilon$ for all $x\in\mathbb Q$, and that is not contained in the inverse image of a compact set.
\end{example}

\begin{example}[Continuation of Example~\ref{E:Z}] If $\mathcal I$ is a compatible superequivalence on an algebra $A$, then $Z(\mathcal I)$ is a compatible superuniformity on $A$.
\end{example}

\subsection{Compatible superuniformities of congruence-permutable algebras}

As with algebras with a compatible uniformity, or with a compatible superequivalence, Mal'tsev's Theorem generalizes:

\begin{theorem}
Let $A$ be an algebra with a Mal'tsev term $p$. Then
any two compatible superuniformities on $A$ permute.
\end{theorem}

\begin{proof}
Let $\mathcal E$, $\mathcal E'$ be compatible superuniformities on $A$.
It suffices to show $\mathcal E\circ\mathcal E'\leq\mathcal E'\circ\mathcal E$. Thus it suffices to show, if $\mathcal F\in\mathcal E$, $\mathcal F'\in\mathcal E'$, then $\mathcal F\circ\mathcal F'\in\mathcal E'\circ\mathcal E$.

$\Fg\{\,\Delta\,\}$ belongs to every superuniformity, by $\mathrm{(SU_r)}$.
We have
\[p(\Fg\{\,\Delta\,\},\Fg\{\,\Delta\,\},\mathcal F')\circ p(\mathcal F,\Fg\{\,\Delta\,\},\Fg\{\,\Delta\,\})\in\mathcal E\]
since $\mathcal E$ is compatible and satisfies $\mathrm{(SU_t)}$.

If $X\in p(\Fg\{\,\Delta\,\},\Fg\{\,\Delta\,\},\mathcal F')\circ
p(\mathcal F,\Fg\{\,\Delta\,\},\Fg\{\,\Delta\,\})$, then
for some $F\in\mathcal F$, $F'\in\mathcal F'$, we have
\[p(\Delta,\Delta,F')\circ p(F,\Delta,\Delta)\subseteq X;\]
$a\mathrel F b\mathrel F' c$ then implies
\[a=p(a,b,b)\mathrel {p(\Delta,\Delta,F')} p(a,b,c)
\mathrel{p(F,\Delta,\Delta)}p(b,b,c)=c;\]
thus,
$\mathcal F\circ\mathcal F'\leq  p(\Fg\{\,\Delta\,\},\Fg\{\,\Delta\,\},\mathcal F')\circ
p(\mathcal F,\Fg\{\,\Delta\,\},\Fg\{\,\Delta\,\})$ and $\mathcal F\circ\mathcal F'\in\mathcal E'\circ\mathcal E$.
\end{proof}

\begin{corollary}
Let $A$ be an algebra with a Mal'tsev term.  Then $\OpSuperUnif A$ is modular.
\end{corollary}

\subsection{Compatible superuniformities on algebras in congruence-modular
Varieties}\label{S:CMSU}

The above results show  that in the case of congruence-permutable varieties $\mathbf V$, the lattices $\OpSuperUnif A$ behave similarly to the lattices $\Con A$, $\OpUnif A$, and $\OpSuperEqv A$. 
Now we ask whether $\OpSuperUnif A$ is modular when
$A\in\mathbf V$, when $\mathbf V$ satisfies the weaker condition of congruence-modularity.

\begin{notation} Given a finite tuple $\mathbf m$ of quaternary terms, and given a filter $\mathcal F$ of relations on an algebra $A$, we define
\[\mathbf m(\mathcal F)=\Fg\{\,\mathbf m(R,R,R,R)\mid R\in\mathcal F\,\},\]
where
\begin{align*}
\mathbf m(R)&=\bigcup_im_i(R,R,R,R)\\
&=\bigcup_i\{\,\pair{m_i(a,b,c,d)}{m_i(a',b',c',d')}\mid
a\mathrel Ra',b\mathrel Rb',c\mathrel Rc',\text{ and }d\mathrel Rd'\,\}.
\end{align*}\end{notation}

\begin{lemma}\label{T:mofRSU}
$\mathcal F\leq\mathbf m(\mathcal F)$.
\end{lemma}

\begin{proof} $m_i(a,a,a,a)=a$ for all $i$ and all $a\in A$.
\end{proof}

For this discussion, $\mathbf V$ will be a
congruence-modular variety, $m_i$, $i=0$, $\ldots$, $\D$
will be a sequence of Day terms for $\mathbf V$, and $A$
will be an algebra in $\mathbf V$. Our first goal will again be
 a generalization of the Shifting Lemma
\cite{Gumm}.

\begin{lemma} \label{T:Lemma1}
Let $\mathcal E\in\OpSuperUnif A$. If $\mathcal W\in\mathcal E$, and $Y\in\mathbf m(\mathcal W')^{2\D}$, the relational composition of $2\D$ copies of the filter $\mathbf m(\mathcal W')$,  where $\mathcal W'=\mathcal W\vee\mathcal W^{\mathrm{op}}$, then there is a $W\in\mathcal W$ such that if 
$a$, $b$, $c$, $d\in A$ with $b\mathrel{W}d$ and
$m_i(a,a,c,c)\mathrel{W}m_i(a,b,d,c)$ for all $i$,
then we have
$a\mathrel Yc$.
\end{lemma}

\begin{proof}
If $Y\in\mathbf m(\mathcal W')^{2\D}$, then there are reflexive, symmetric $W_1$, $\ldots$, $W_{2\D}\in\mathcal W'$ such that
$\mathbf m(W_1)\circ\ldots\circ\mathbf m(W_{2\D})\subseteq Y$.
Let $W=\bigcap_j W_j$ and note that if $b\mathrel Wd$, then $b\mathrel{\mathbf m(W)}d$ by Lemma~\ref{T:mofRSU}.

Let $u_i=m_i(a,b,d,c)$ and $v_i=m_i(a,a,c,c)$ for all $i$.
If $b\mathrel Wd$ and $m_i(a,a,c,c)\mathrel Wm_i(a,b,d,c)$ for all $i$, then for even $i<\D$, we have
$u_i\mathrel{\mathbf m(W)}v_i=v_{i+1}\mathrel{\mathbf m(W)}u_{i+1}$, while for  odd $i<\D$, we have
$u_i\mathrel{\mathbf m(W)}m_i(a,b,b,c)=m_{i+1}(a,b,b,c)\mathrel{\mathbf m(W)}u_{i+1}$.

Thus, for $i<\D$, $u_i\mathrel{\mathbf m(W)\circ\mathbf m(W)}u_{i+1}$, and consequently, $a=u_0\mathrel Yu_\D=c$.
\end{proof}

\begin{lemma}[Shifting Lemma for Superuniformities] \label{T:Lemma4}
Let $\mathcal R\in\Idl\Fil\Rel A$ satisfy $\mathrm{(SU_r)}$ and $\mathrm{(SU_s)}$ and be compatible, and let
$\mathcal E_1$,
$\mathcal E_2\in\OpSuperUnif A$ be such that
$\mathcal R\wedge\mathcal E_1\leq\mathcal E_2\leq\mathcal E_1$. Then
$(\mathcal R\circ(\mathcal E_1\wedge\mathcal E_2)\circ\mathcal
R)\wedge\mathcal E_1\leq\mathcal E_2$.
\end{lemma}

\begin{proof} 
It suffices to show that, given filters $\mathcal G\in\mathcal R$, $\mathcal F\in\mathcal E_1$, and $\mathcal X\in\mathcal E_1\wedge\mathcal E_2=\mathcal E_1\cap\mathcal E_2$, there is a filter $\mathcal Y\in\mathcal E_2$ such that
\[(\mathcal G\circ\mathcal X\circ\mathcal G)\wedge\mathcal F\leq\mathcal Y;\]
to prove this, it suffices to construct $\mathcal Y$ and show that given $Y\in\mathcal Y$, there are $G\in\mathcal G$, $X\in\mathcal X$, and $F\in\mathcal F$ such that $(G\circ X\circ G)\cap F\subseteq Y$, or, in other words, that $a\mathrel Gb\mathrel Xd\mathrel Gc$ and $a\mathrel Fc$ imply $a\mathrel Yc$.

Without loss of generality, we may assume that the given $\mathcal G$, $\mathcal F$, and $\mathcal X$ are symmetric (i.e., satisfy $\mathrm{(U_s)}$), because $\mathcal R$, $\mathcal E_1$, and $\mathcal E_2$ all satisfy $\mathrm{(SU_s)}$.

Let $\mathcal W=\mathbf m(\mathcal X)\vee(\mathbf m(\mathcal G)\wedge(\mathbf m(\mathcal F)\circ\mathbf m(\mathcal F)\circ\mathbf m(\mathcal X)))$. Since $\mathcal R$ is compatible, $\mathcal E_1$ and $\mathcal E_2$ are compatible superuniformities, and $\mathcal R\wedge\mathcal E_1\leq\mathcal E_2$, we have $\mathcal W\in\mathcal E_2$.

Let $\mathcal Y=\mathcal W^{2\D}$ and let $Y\in\mathcal Y$. By Lemma~\ref{T:Lemma1}, there is a $W\in\mathcal W$ such that $b\mathrel Wd$ and
$m_i(a,a,c,c)\mathrel Wm_i(a,b,d,c)$ for all $i$ imply $a\mathrel Yc$.

We have $\mathbf m(X)\cup(\mathbf m(G)\cap(\mathbf m(F)\circ\mathbf m(F)\circ\mathbf m(X)))\subseteq W$ for some $G\in\mathcal G$, $F\in\mathcal F$, and $X\in\mathcal X$.
Since $\mathcal G$ and $\mathcal F$ satisfy $\mathrm{(U_s)}$, we may assume $G$ and $F$ are symmetric, replacing them by $G\cap G^{\mathrm{op}}$ and $F\cap F^{\mathrm{op}}$ if necessary.

If $a\mathrel Gb\mathrel Xd\mathrel Gc$ and $a\mathrel Fc$, then
$m_i(a,a,c,c)\mathrel{\mathbf m(G)}m_i(a,b,d,c)$ for all $i$, and $m_i(a,a,c,c)\mathrel{\mathbf m(F)}m_i(a,a,a,a)=a=m_i(a,b,b,a)\mathrel{\mathbf m(F)}m_i(a,b,b,c)\mathrel{\mathbf m(X)}m_i(a,b,d,c)$ for all $i$. Thus, $m_i(a,a,c,c)\mathrel Wm_i(a,b,d,c)$ for all $i$. We also have $b\mathrel Wd$ because $b\mathrel Xd$.
Then by Lemma~\ref{T:Lemma1}, $a\mathrel Yc$.
\end{proof}

Finally, we have

\begin{theorem}\label{T:ModularSU} The lattice $\OpSuperUnif A$ is modular.
\end{theorem}

\begin{proof}
Let $\mathcal E$, $\mathcal E'$,
$\mathcal E''\in\OpSuperUnif A$ be such
that $\mathcal E\leq\mathcal E''$. Then
\[\mathcal E\vee(\mathcal E'\wedge\mathcal E'')
\leq(\mathcal E\vee\mathcal E')\wedge\mathcal E'',
\]
as this inequality holds in every
lattice when $\mathcal E\leq\mathcal E''$.

Define $\mathcal R_k$ and $\mathcal R'_k$ for every natural number $k$ as $\mathcal
R_0=\mathcal R'_0=\mathcal E'$, $\mathcal R_{k+1}=\mathcal R_k\circ\mathcal E\circ\mathcal R_k$,
and $\mathcal R'_k=\mathcal R_k\circ(\mathcal E_1\wedge\mathcal E_2)\circ\mathcal R_k$ where $\mathcal E_1=\mathcal E''$ and $\mathcal E_2=\mathcal E\vee(\mathcal E'\wedge\mathcal E'')$. Since $\mathcal E\leq\mathcal E_2\leq\mathcal E_1$ we have for every $k$, $\mathcal R_k\leq\mathcal R'_k$. We also have $\mathcal R'_0\wedge\mathcal E_1=\mathcal E'\wedge\mathcal E''\leq\mathcal E_2$. Then by Lemma~\ref{T:Lemma4} and induction on $k$, we have
$\mathcal R'_k\wedge\mathcal E''=\mathcal R'_k\wedge\mathcal E_1\leq\mathcal E_2$ for all $k$. Since the lattice $\Idl\Fil\Rel A$ is algebraic,
we have
\begin{align*}
(\mathcal E\vee\mathcal E')\wedge\mathcal E''
&=\left(\bigcup_k\mathcal R_k\right)\wedge\mathcal E_1\\
&\leq\left(\bigcup_k\mathcal R'_k\right)\wedge\mathcal E_1\\
&=\bigcup_k\left(\mathcal R'_k\wedge\mathcal E_1\right)\\
&\leq\mathcal E_2=\mathcal E\vee(\mathcal E'\wedge\mathcal E'').
\end{align*}
\end{proof}

\begin{remark}
Theorem~\ref{T:ModularSU} and the embedding of Example~\ref{E:Z} give us an alternate proof of Theorem~\ref{T:ModularSE}.
\end{remark}

\section{Pattern}\label{S:Pattern}

After seeing Sections~\ref{S:Unif}, \ref{S:SuperEquivalences}, and~\ref{S:SuperUniformities}, including the theorems about modularity in Section~\ref{S:SuperEquivalences} and~\ref{S:SuperUniformities}, the reader may suspect that might be a general category-theoretic setting in which the modularity results can be proved, when, as in the case of superequivalences and super\-uniformities, there are suitable algebraicity assumptions.  Indeed, we are preparing a paper (\cite{rgeneq}) about such a setting.

\section{Conclusions}\label{S:Conclu}

The reader may also ask reasonably, are superequivalences and superuniformities useful? The answer is that they arise naturally in the theory of the categories $\Ind(\Set)$ and $\Ind(\CatFil)$, where $\Ind(\mathbf C)$ for a category $\mathbf C$ is the category of small filtered diagrams in the category $\mathbf C$ (see \cite[Section~IX.1]{macl} for the definition of a filtered category; a small filtered diagram is a functor $D:\mathbf D\to\mathbf C$ where $\mathbf D$ is small and filtered), and $\CatFil$ is the category of filters and germs of admissible partial functions, as defined (with very minor differences to account for some boundary cases) in \cite{kr70}. $\Ind(\Set)$ and $\Ind(\CatFil)$ are both cartesian-closed categories, and we believe $\Ind(\CatFil)$ will have an important role in Topological Algebra because, in addition to being cartesian-closed, and unlike the category of compactly-generated spaces advocated by Mac Lane \cite{macl} and others, algebra objects in congruence-modular varieties have modular quotient lattices. Of course, the reader needs details and these will be forthcoming in future papers.

\begin{comment}
\subsection*{Acknowledgement}
We would like to thank Keith Kearnes for his ideas on
this topic, and some lively discussions.
\end{comment}

\bibliographystyle{abbrv}
\bibliography{../Bibliographies/mybib}

\end{document}